\documentclass[11pt,a4paper,reqno]{amsart}
\usepackage{amsmath}
\usepackage{amsfonts}
\usepackage{amssymb}
\usepackage{amscd}
\usepackage{url}
\usepackage[pdftex,bookmarks=true]{hyperref}

\newcommand\Vol{{\operatorname{Vol}}}
\newcommand\rank{{\operatorname{rank}}}

\newcommand\R{{\mathbf{R}}}

\renewcommand\P{{\mathbf{P}}}
\newcommand\E{{\mathbf{E}}}

\newcommand\Z{{\mathbf{Z}}}

\newcommand\col{{\mathbf{c}}}
\newcommand\row{{\mathbf{r}}}

\newcommand\eps{\varepsilon}

\newcommand\ep{\epsilon}

\theoremstyle{plain}
  \newtheorem{theorem}[subsection]{Theorem}
  \newtheorem{conjecture}[subsection]{Conjecture}

  \newtheorem{lemma}[subsection]{Lemma}
  
  \newtheorem{question}[subsection]{Question}
  \newtheorem{example}[subsection]{Example}

\theoremstyle{remark}
  \newtheorem{remark}[subsection]{Remark}

\theoremstyle{definition}
  \newtheorem{definition}[subsection]{Definition}

\include{psfig}

\begin{document}

\title[Inverse Littlewood-Offord Problems]{ Inverse Littlewood-Offord problems and the singularity of random symmetric matrices}

\author{Hoi H. Nguyen}
\address{Department of Mathematics, University of Pennsylvania, 209 South 33rd Street, Philadelphia, PA 19104, USA}
\email{hoing@math.upenn.edu}

\begin{abstract} Let $M_n$ denote a random symmetric $n$ by $n$ matrix, whose upper diagonal entries are iid Bernoulli random variables (which take value $-1$ and $1$ with probability 1/2). Improving the earlier result by Costello, Tao and Vu \cite{CTV}, we show that $M_n$ is non-singular with probability $1-O(n^{-C})$ for any positive constant $C$. The proof uses an inverse Littlewood-Offord result for quadratic forms, which is of interest of its own. 
\end{abstract}

\maketitle

\section{Introduction}\label{section:introduction}

Let $A_n$ denote a random $n$ by $n$ matrix, whose entries are iid Bernoulli random
variables which take values $\pm 1$ with probability 1/2. Let $p_n$ be the probability that $A_n$ is singular. A classical result of Koml\'os \cite{B,K} shows

\begin{equation}\label{eqn:Komlos}
p_n=O(n^{-1/2}).
\end{equation}

By considering the event that two rows or two columns of $A_n$ are equal (up to a sign), it is clear that 

$$p_n\ge (1+o(1))n^2 2^{1-n}.$$ 

It has been conjectured by many researchers that in fact this bound is best possible.

\begin{conjecture}\label{conjecture:1}

$$p_n=(\frac{1}{2}+o(1))^n.$$

\end{conjecture}

In a breakthrough paper, Kahn, Koml\'os and Szemer\'edi \cite{KKSz} proved that

$$p_n= O(.999^n).$$

Another significant improvement is due to Tao and Vu \cite{TVsing}, who used inverse theory from additive combinatorics to show that  $p_n=O((3/4)^n)$. The most recent record is due to Bourgain, Vu and Wood \cite{BVW}, who improved it to $p_n=O((1/\sqrt{2})^n)$.

Another popular model of random matrices is that of random symmetric matrices;
this is one of the simplest models that has non-trivial correlations between the matrix
entries. Let $M_n$ denote a random symmetric $n$ by $n$ matrix, whose upper diagonal
entries are iid Bernoulli random variables. 

Let $q_n$ be the probability that $M_n$ is singular. Despite its obvious similarity to $p_n$, less is known concerning the bound for $q_n$. A significant new difficulty is that the symmetry ensures that the determinant $\det(M_n)$ is a quadratic function of each row, as opposed to $\det(A_n)$ which is a linear function of each row. 

As far as we can trace, the question to determine whether $q_n$ tends to zero together with $n$ was first posed by Weiss in the early nineties. This simple looking question had been open until a recent breakthrough paper by Costello, Tao and Vu \cite{CTV}, who showed

$$q_n= n^{-1/8+o(1)}.$$

To prove this result, Costello, Tao and Vu introduced and studied a {\it quadratic} variant of the classical Erd\H{o}s-Littlewood-Offord inequality concerning the concentration of random variables. Note that this classical inequality plays a key role in the work of Koml\'os to establish \eqref{eqn:Komlos}.   
 
Although the bound $q_n=n^{-1/8+o(1)}$ can be improved further by applying the more recent inequalities from \cite{C}, it seems that the approach developed by Costello, Tao and Vu cannot give any bound better than $n^{-1/2+o(1)}$.

In this paper we show that $q_n$ decays faster than any polynomial in $n$.

\begin{theorem}[Main theorem] \label{theorem:singularity} We have
$$q_n= O(n^{-C})$$ 

for any positive constant $C$, where the implied constant depends on $C$. 
\end{theorem}

One may hope to combine our approach and the ''replacement technique'' from \cite{KKSz} and \cite{TVsing} to improve the bound further to exponential decay. However, we have not been able to do so. It is commonly believed that (see \cite{V})

\begin{conjecture}
$$q_n=(\frac{1}{2}+o(1))^n.$$
\end{conjecture}

{\bf Notation}. Here and later, asymptotic notations  such as $O, \Omega, \Theta$, and so for, are used under the assumption that $n \rightarrow \infty$. A notation such as $O_{C}(.)$ emphasizes that the hidden constant in $O$ depends on $C$. If $a= \Omega (b)$, we write $b \ll a$ or $a \gg b$. 

For a matrix $A$ we use the notations $\row_i(A)$ and $\col_j(A)$ to denote its $i$-th row and $j$-th column respectively; we use the notation $A(ij)$ to denote its $ij$ entry.

\section{The approach} Let $x=(x_1,\dots,x_n)$ be the first row of $M_{n}$, and $a_{ij}, 2\le i,j\le n$, be the cofactors of $M_{n-1}$ obtained by removing $x$ and $x^T$ from $M_n$. We have 

$$\det(M_n)= x_1^2 \det(M_{n-1}) + \sum_{2\le i,j \le n} a_{ij}x_ix_j.$$

Roughly speaking, the main approach of \cite{CTV} is to show that with high probability (with respect to $M_{n-1}$) most of the $a_{ij}$ are nonzero. It then follows that, via the so called quadratic Littlewood-Offord inequality (Theorem \ref{theorem:quadratic:LOinequality}), 

$$\P_x(\det(M_n)=0) =n^{-1/8+o(1)}.$$  

In this paper we adapt the reversed approach, which consists of two main steps outlined below.

\begin{enumerate}

\item If $\P_x(\det(M_n)=0|M_{n-1}) \ge n^{-O(1)}$, then there is a strong additive structure among the cofactors $a_{ij}$.

\vskip .1in 
  
\item With respect to $M_{n-1}$, a strong additive structure among the $a_{ij}$ occurs with negligible probability.   
\end{enumerate}

The first step, which is at the heart of our paper, concentrates on the study of inverse Littlewood-Offord problem for linear forms and quadratic forms. We will provide an almost complete answer to this problem throughout Section \ref{section:linearILO}, \ref{section:bilinearILO}, and \ref{section:quadraticILO}.

For the rest of this section we sketch the proof of Theorem \ref{theorem:singularity}.

We first show that it is enough to consider the case of $M_n$ having rank $n-1$, thanks to the following result.

\begin{lemma}\label{lemma:rankincrease}  For any $1\le k \le n-2$,
$$\P(\rank(M_n)=k \le n-2) \le  0.1 \times \P(\rank (M_{2n-k-1})=2n-k-2).$$
\end{lemma}

We deduce Lemma \ref{lemma:rankincrease} from a useful observation by Odlyzko. 

\begin{lemma}[Odlyzko's lemma,\cite{O}]\label{lemma:Odlyzko}
Let $H$ be a linear subspace in $\R^n$ of dimension at most $k\le n$. Then it contains at most $2^k$ vectors from $\{-1,1\}^n$.
\end{lemma}

\begin{proof}(of Lemma \ref{lemma:rankincrease}) Because $M_n$ has rank $k$, the  subspace spanned by its rows intersects $\{-1,1\}^n$ in a set $H$ of no more than $2^k$ vectors. Thus the probability that the subvector formed by the last $n$ components of the first row of $M_{n+1}$ does not belong to $H$ is at least $1-2^{-n+k}$. Hence, 

$$\P(\rank(M_{n+1})= k+2 |\rank(M_n)= k) \ge 1-2^{-n+k}.$$

In general, for $1\le t\le n-k$ we have

$$\P(\rank(M_{n+t}) = k+2t | \rank(M_{n+t-1})= k+2(t-1)) \ge 1-2^{-n+t+k-1}.$$
 
Because the rows (and columns) added to $M_{n+t-1}$ each step (to create $M_{n+t}$) are independent, we have

\begin{align*} 
&\qquad \P(\rank(M_{2n-k-1})=2n-k-2 |\rank(M_n)=k) \ge\\ 
&\ge \prod_{t=1}^{n-k-1} \P\big(\rank(M_{n+t}) = k+2t |\rank(M_{n+t-1})= k+2(t-1)\big)  \\
&\ge (1-2^{-n+k})(1-2^{-n+k+1})\dots (1-2^{-1})\ge 0.1.
\end{align*}

\end{proof}

Next we show that in the case of $M_n$ having rank $n-1$, it suffices to assume that $\rank(M_{n-1})\ge n-2$, thanks to the following simple observation.

\begin{lemma}\label{lemma:optional}
Assume that $M_n$ has rank $n-1$. Then there exists $1\le i \le n$ such that the removal of the $i$-th row and the $i$-column of $M_{n}$ results in a symmetric matrix $M_{n-1}$ of rank at least $n-2$.
\end{lemma}

\begin{proof} (of Lemma \ref{lemma:optional})
With out loss of generality, assume that the last $n-1$ rows of $M_n$ span a subspace of dimension $n-1$. Then the matrix obtained from $M_n$ by removing the first row and the first column has rank at least $n-2$.
\end{proof}

To prove Theorem \ref{theorem:singularity}, it thus suffices to prove

\begin{theorem}\label{theorem:singularity:case1}
$$\P(\det(M_n)=0, \rank(M_{n-1})=n-1)=O(n^{-C}).$$
\end{theorem}

\begin{theorem}\label{theorem:singularity:case2}
$$\P(\det(M_n)=0, \rank(M_{n-1})=n-2)= O(n^{-C}).$$
\end{theorem}

\vskip .1in

We will prove Theorem \ref{theorem:singularity:case1} by relying on a structural lemma stated below, which follows from our study of the inverse Littlewood-Offord problem for linear forms in Step 1.

\begin{lemma}[Structural theorem, degenerate case]\label{lemma:normalvector:linear} Let $\ep<1$ and $C$ be positive constants. Assume that $M_{n-1}$ has rank  $n-2$ and that

$$\P_x(\sum_{i,j}a_{ij}x_ix_j=0|M_{n-1})\ge n^{-C}.$$ 

Then there is a nonzero vector $u=(u_1,\dots,u_{n-1})$ with the following properties.

\begin{itemize}
\item All but $n^\ep$ elements of $u_i$ belong to a symmetric proper generalized arithmetic progression of rank $O_{C,\ep}(1)$ and size $n^{O_{C,\ep}(1)}$. 

\vskip .1in

\item $u_i \in \{p/q: p,q\in \Z, |p|,|q|=n^{O_{C,\ep}(n^\ep)}\}$ for all $i$.

\vskip .1in

\item $u$ is orthogonal to $n-O_{C,\ep}(n^\ep)$ rows of $M_{n-1}$.
\end{itemize}

\end{lemma}

We refer the reader to Section \ref{section:linearILO} for a definition of generalized arithmetic progression. Theorem \ref{theorem:singularity:case2} follows from a similar structural lemma, which can be deduced from our study of the inverse Littlewood-Offord problem for quadratic forms in Step 1.

\begin{lemma}[Structural theorem, non-degenerate case]\label{lemma:normalvector:quadratic} Let $\ep<1$ and $C$ be positive constants.  Assume that $M_{n-1}$ has rank $n-1$ and that

$$\P_x(\sum_{i,j}a_{ij}x_ix_j=0|M_{n-1})\ge n^{-C}.$$ 

Then there exists a nonzero vector $u=(u_1,\dots,u_{n-1})$ with the following properties.

\begin{itemize}
\item All but $n^\ep$ elements of $u_i$ belong to a proper symmetric generalized arithmetic progression of rank $O_{C,\ep}(1)$ and size $n^{O_{C,\ep}(1)}$. 

\vskip .1in

\item $u_i \in \{p/q: p,q\in \Z, |p|,|q|=n^{O_{C,\ep}(n^\ep)}\}$ for all $i$.

\vskip .1in

\item $u$ is orthogonal to $n-O_{C,\ep}(n^\ep)$ rows of $M_{n-1}$.
\end{itemize}

\end{lemma}

The rest of the paper is organized as follows. In Sections \ref{section:linearILO}-\ref{section:quadraticILO}, we discuss the inverse Littlewood-Offord problem in details. As applications, we prove Lemma \ref{lemma:normalvector:linear} and Lemma \ref{lemma:normalvector:quadratic} in Section \ref{section:normalvector:linear} and Section \ref{section:normalvector:quadratic} respectively. We conclude by proving Theorem \ref{theorem:singularity:case1} and Theorem \ref{theorem:singularity:case2} in Section \ref{section:singularity:completion}.

\section{The inverse Littlewood-Offord problem for linear forms}\label{section:linearILO}

Let $x_i, i=1,\dots, n$ be iid Bernoulli random variables, taking values $\pm 1$ with probability $\frac{1}{2} $. Given a multiset $A$ of $n$ real number  $a_1, \dots, a_n$, we define the random walk $S$ with steps in $A$ to be the random variable $S:=\sum_{i=1}^n{a_ix_i}$. The \emph{concentration probability} is defined to be

$$\rho(A) := \sup_{a} \P(S=a). $$

Motivated by their study of random polynomials, in the 1940s Littlewood and Offord \cite{LO} raised the question of bounding $\rho(A)$. (We call this the  {\it forward} Littlewood-Offord problem, in contrast with the {\it inverse} Littlewood-Offord problem discussed later.) They showed that if the $a_i$ are nonzero then $\rho(A)=O(n^{-1/2}\log n)$. Shortly after the Littlewood-Offord paper, Erd\H{o}s \cite{E} gave a beautiful combinatorial proof of the refinement

\begin{equation} \label{eqn:Erdos} \rho(A)\le \frac{ \binom{n}{n/2}}{2^n } =O(n^{-1/2}). \end{equation}

\vskip .1in

The results of Littlewood-Offord and Erd\H{o}s are classics  in combinatorics and have  generated an impressive wave of research, particularly from the early 1960s to the late 1980s.

One direction of research was to generalize Erd\H os' result  to
other groups. For example, in  1966 and  1970 \cite{Kle}, Kleitman extended Erd\H{o}s' result to
complex numbers and normed vectors, respectively. Several results in this direction can be found in \cite{Kat}.

Another direction was  motivated by the observation that \eqref{eqn:Erdos} can be improved significantly by making additional assumptions about $V$. The first such result was discovered by Erd\H{o}s and Moser \cite{EM},
who showed that if $a_i$ are distinct, then $\rho(A) = O(n^{-3/2} \log n)$. This bound was then sharpened to $\rho(A)=O(n^{-3/2})$ by S\'ark\H{o}zy and Szemer\'edi \cite{SSz}. Another famous result regarding this result of Erd\H os and Moser is that of Stanley \cite{Stan}, who shows that if $a_i$ are distinct then $\rho(A)\le \rho(A_0)$, where $A_0 :=\{- \lfloor n/2 \rfloor , \dots, \lfloor n/2 \rfloor \}$.

In \cite{H} (see also in \cite{TVbook}),  Hal\'asz
proved very general theorems that imply the S\'ark\H{o}zy-Szemer\'edi theorem and many others.
One of his results can be formulated as follows.

\begin{theorem} \label{theorem:Hal}
Let $l$ be a fixed integer and $R_l$ be the number of solutions of the equation
$a_{i_1}+\dots+a_{i_l}=a_{j_1}+\dots+a_{j_l}$. Then

$$\rho(A)= O(n^{-2l-\frac{1}{2}}R_l).$$
\end{theorem}

We remark that the Erd\H{o}s-Littlewood-Offord inequality \eqref{eqn:Erdos} and Theorem \ref{theorem:Hal} of Hal\'asz can be extended to the continuous setting. This type of concentration has been vastly investigated in the literature, we refer the reader to \cite{Ess,H,Kol,Rog} for further reading. We mention here an asymptotic result of Kanter \cite{Kan}, which generalizes \eqref{eqn:Erdos} and is closely related to our discussion 

\begin{theorem}\label{theorem:Kanter} Let $\Phi$ be a symmetric convex measurable set in a vector space $V$, and $a_i\in V$. Assume that there are $\Theta(n)$ indices $i$ such that $a_i\notin \Phi$. Then we have

$$\sup_{a} \P(S\in a+\Phi) = O(n^{-1/2}).$$
  
\end{theorem}

Let us now turn to the main goal of this section. 

Motivated by inverse theorems from additive combinatorics (see \cite[Chapter 5]{TVbook}) and a variant for random sums in \cite[Theorem 5.2]{TVsing}, Tao and Vu \cite{TVinverse} brought a different view to the problem. Instead of trying to improve the bound further by imposing new assumptions as done in the forward problems, they tried  to provide the complete picture by finding the underlying reason as to why the concentration probability is large (say, polynomial in $n$). 

\vskip .1in

Note that the (multi)-set  $A$ has $2^n$ subsums, and $\rho({A})
\ge n^{-C}$ means that at least $\frac{2^n}{n^C}$ among these take the same
value. This observation suggests that  the set should have a very strong additive structure. To determine this structure, let us introduce an important concept in additive combinatorics, \emph{generalized arithmetic progressions} (GAPs). 

A set $Q$ is a \emph{GAP of rank $r$} if it can be expressed as in the form
$$Q= \{g_0+ m_1g_1 + \dots +m_r g_r| N_i \le m_i \le N_i' \hbox{ for all } 1 \leq i \leq r\}$$ for some $g_0,\ldots,g_r,N_1,\ldots,N_r,N'_1,\ldots,N'_r$.

\vskip .1in

It is convenient to think of $Q$ as the image of an integer box $B:= \{(m_1, \dots, m_r) \in \Z^r| M_i \le m_i \le M_i' \} $ under the linear map
$$\Phi: (m_1,\dots, m_r) \mapsto g_0+ m_1g_1 + \dots + m_r g_r. $$
The numbers $g_i$ are the \emph{generators } of $P$, the numbers $N_i',N_i$ are the \emph{dimensions} of $P$, and $\Vol(Q) := |B|$ is the \emph{volume} of $B$. We say that $Q$ is \emph{proper} if this map is one to one, or equivalently if $|Q| = \Vol(Q)$.  For non-proper GAPs, we of course have $|Q| < \Vol(Q)$.
If $-N_i=N_i'$ for all $i\ge 1$ and $g_0=0$, we say that $Q$ is {\it symmetric}.

We next consider an example of $A$ where $\rho(A)$ is large. For a positive integer $l$ we denote the set $ \{a_1+ \dots + a_l| a_i \in A\}$ by $lA$.

\vskip .2in

\begin{example}[Structure implies large concentration probability] \label{example:linear:GAP} Let $Q$ be a proper symmetric GAP of rank $r$ and volume $N$. Let $a_1, \dots, a_n$ be (not necessarily distinct) elements of
$P$. The random variable $S =\sum_{i=1}^n a_ix_i$ takes values in
the GAP $nP$. Because $|nP| \le \Vol (nB) = n^r N$, the pigeonhole
principle implies that $\rho(V)   \ge \Omega (\frac{1}{n^r N})$. In
fact, by using the second moment method, one can
improve the bound to $\Omega (\frac{1}{n^{r/2} N})$. If we set
$N=n^{C-r/2}$ for some constant $C\ge r/2$, then

\begin{equation} \label{bound2} \rho(V)   = \Omega (\frac{1}{n^{C}}).
\end{equation}
\end{example}

The example above shows that, if the elements of $A$
belong to a symmetric proper GAP with a small rank and small cardinality, then
$\rho(V)$ is large. A few years ago, Tao and Vu  \cite{TVstrong,TVinverse}
proved several versions showing that this is essentially the only reason. We present here an optimal version due to Vu and the current author.

\begin{theorem}[Optimal inverse Littlewood-Offord theorem for linear forms] \cite[Theorem 2.5]{NgV} \label{theorem:ILO} Let $\ep<1$ and $C$ be positive constants. Assume that

$$\rho (A)  \ge  n^{-C}. $$

Then, for any $n^\ep \le n' \le n $, there exists a proper symmetric GAP $Q$ of rank $r= O_{C, \eps} (1)$ that contains all but at most $n'$ elements of $A$ (counting multiplicity), where

 $$|Q| = O_{C, \eps} ( \rho (A)^{-1} {n'}^{- \frac{r}{2}}). $$
\end{theorem}

Our method can be extended to more general distributions. We just cite one below for our later applications. 

Let $0<\mu \le 1$ be a positive parameter. Let  $\eta^\mu$ be a random variable such that $\eta^\mu=1$ or $-1$ with probability $\mu/2$,
and $\eta^\mu=0$ with probability $1-\mu$.

\vskip .1in

\begin{theorem}\label{theorem:ILO:general}
The conclusion of Theorem \ref{theorem:ILO} also holds if the $x_i$ are iid copies of $\eta^\mu$.  
\end{theorem}

\vskip .1in

\begin{remark} In their work to obtain the bound $p_n=O((3/4)^n)$, Tao and Vu studied a similar inverse problem. 

Let $0<\mu<1/4$ be a parameter, and let $\ep<1$ be a positive constant. 

Define 

$$\rho^{(\mu)}(A):=\sup_{a\in \R}\P(\sum_{i=1}^n a_i\eta_i^\mu =a).$$ 

It can be shown that $\rho(A)\le \rho^{(\mu)}(A)$. In \cite{TVsing}, Tao and Vu characterized those $A$ where $\rho(A)$ is comparable to $\rho_\mu(A)$,

$$\rho(A)\ge \ep \rho^{(\mu)}(A).$$

\end{remark}

\section{The inverse  Littlewood-Offord problem for bilinear forms}\label{section:bilinearILO}

Let $x_i,y_j$ be iid Bernoulli random variables, let $A=(a_{ij})$ be an $n\times n$ matrix of real entries. We define the {\it bilinear concentration probability} of $A$ by

$$\rho_b(A):= \sup_{a \in \R}\P(\sum_{i,j} a_{ij}x_iy_j=a).$$

More generally , if $x_i,y_j$ are iid copies of $\eta^\mu$, then the {\it weighted bilinear concentration probability} of $A$ is defined by

$$\rho_b^{(\mu)}(A)=\sup_{a \in \R}\P(\sum_{i,j} a_{ij}x_iy_j=a).$$

As an application of the Littlewood-Offord-Erd\H{o}s inequality \eqref{eqn:Erdos}, it has been shown in \cite{C} (also in \cite{CTV} with a weaker bound) that

\begin{theorem}[Bilinear Littlewood-Offord inequality]\label{theorem:bilinear:LOinequality}
Suppose that there are $\Theta(n)$ indices $i$ such that for each $i$ there are $\Theta(n)$ indices $j$ such that $a_{ij}\neq 0$. Then

$$\rho_b(A)=O(n^{-1/2}).$$
\end{theorem}

The bound $O(n^{-1/2})$ is sharp, as the bilinear form $\sum_{i,j} x_iy_j$ shows.

The bilinear Littlewood-Offord inequality for the continuous setting was also studied in the literature. For instance, as an application of Kanter's inequality (Theorem \ref{theorem:Kanter}), it follows from a result of Rosi\'nski and Samorodnitsky \cite{RS} that 

\begin{theorem}  Let $\Phi$ be a symmetric convex measurable set in a vector space $V$, and $a_i\in V$. Assume that there are $\Theta(n)$ indices $i$ such that for each $i$ there are $\Theta(n)$ indices $j$ such that $a_{ij}\notin \Phi$. Then we have

$$\sup_{a\in V}\P(\sum_{i,j} a_{ij}x_iy_j \in a+\Phi) = O(n^{-1/16}).$$
\end{theorem}

Rosi\'nski and Samorodnitsky also studied concentration inequalities for more general multilinear forms. We refer the reader to \cite{RS} for further reading.

Motivated by the inverse Littlewood-Offord results for linear forms, our goal is to find the reason as to why $\rho_b(A)$ is large.

\begin{question}
Is it true that if $\rho_b(A)$ is large then there must be a ''structural" relation among the entries of $A$? 
\end{question} 

To answer this question, we first consider a few examples of $A$.

\begin{example}[Additive structure implies large concentration probability]\label{example:bilinear:1}
Let $Q$ be a proper symmetric GAP of rank $r=O(1)$ and of size $n^{O(1)}$. Assume that $a_{ij} \in Q$, for all $a_{ij}$. Then 
for any $x_i,y_j \in \{\pm 1\}$,

$$\sum_{i,j}a_{ij}x_iy_j\in n^2Q.$$

Thus, by the pigeon-hole principle, we have 

$$\rho_b(A)\ge n^{-2r}|Q|^{-1} =n^{-O(1)}.$$
   
\end{example}

Our next example shows that if the $a_{ij}$ are ``separable'', then $\rho_b(A)$ is also large. 

\begin{example}[Algebraic structure implies large concentration probability]\label{example:bilinear:2}
Assume that 

$$a_{ij}=k_ib_j + l_jb_i',$$ 

where $b_j,b_i'$ are arbitrary real numbers and $k_i,l_j\in \Z, |k_i|,|l_j|=n^{O(1)}$, such that 

$$\P_{x}(\sum_{i}k_ix_i=0)=n^{-O(1)}$$ 

and 

$$\P_{y}(\sum_{j}l_jy_j=0)=n^{-O(1)}.$$ 

Then we have

$$\P_{x,y}(\sum_{i,j}a_{ij}x_iy_j =0)=\P\left(\sum_{i}k_ix_i\sum_{j}b_jy_j+\sum_{i}b_i'x_i\sum_{j}l_jy_j=0\right)= n^{-O(1)}.$$

\end{example}

\begin{remark}
In the above example, the assumption that $k_i,l_j$ are integers seems unnecessary. However, because $\P_{x}(\sum_{i}k_ix_i=0)=n^{-O(1)}$ and $\P_{y}(\sum_{j}l_jy_j=0)=n^{-O(1)}$, Theorem \ref{theorem:ILO} implies that most of the $k_i$ and $l_j$ belong to a GAP of bounded size. Thus, without loss of generality, we may assume that $k_i,l_j$ are bounded integers.
\end{remark}

Our last example shows that a combination of additive structure and algebraic structure also implies high bilinear concentration probability.

\begin{example}[Structure implies large concentration probability]\label{example:bilinear:3}
Assume that $a_{ij}=a_{ij}' + a_{ij}''$, where $a_{ij}' \in Q$, a proper symmetric GAP of rank $O(1)$ and size $n^{O(1)}$, and  

$$a_{ij}'' = k_{i1}b_{1j}+\dots k_{ir}b_{rj}+ l_{1j}b_{i1}'+\dots+l_{rj}b_{ir}',$$ 

where $b_{1j},\dots, b_{rj}, b_{i1}',\dots, b_{ir}'$ are arbitrary and $k_{i1},\dots, k_{ir}, l_{1j},\dots, l_{rj}$ are integers bounded by $n^{O(1)}$, and $r=O(1)$ such that

$$\P_{x}\left(\sum_i k_{i1}x_i=0,\dots,\sum_i k_{ir}x_i=0\right)=n^{-O(1)}$$ 

and 

$$\P_{y}\left(\sum_{j}l_{1j}y_j=0,\dots, \sum_{j}l_{rj}y_j=0\right)=n^{-O(1)}.$$

Then we have 

\begin{align*}
\sum_{i,j}a_{ij}x_iy_j&= \sum_{i,j}a_{i,j}'x_iy_j + \sum_i k_{i1}x_i \sum_{j} b_{1j}y_j+\dots+\sum_i k_{ir}x_i\sum_{j} b_{rj}y_j\\
&+\sum_i b_{i1}'x_i\sum_{j} l_{1j}y_j +\dots+\sum_i b_{ir}'x_i\sum_{j} l_{rj}y_j.
\end{align*}

Thus, 

$$\P_{x,y}\left(\sum_{i,j}a_{ij}x_iy_j \in n^2Q \right)=n^{-O(1)}.$$ 

It then follows, by the pegion-hole principle, that $\rho_b(A)=n^{-O(1)}$.
\end{example}

The above examples demonstrate that if the $a_{ij}$ can be decomposed into additive and algebraic structural parts, then $\rho_b(A)$ is large. Our inverse result asserts that these are essentially the only ones that have large bilinear concentration probability.

\begin{theorem}[Inverse Littlewood-Offord theorem for bilinear forms]\label{theorem:ILO:bilinear:differentGAP} Let $\ep<1, C$ be positive constants. Assume that

$$\rho_b(A)\ge n^{-C}.$$ 

Then there exist index sets $I_0,J_0$, both of size $O_{C,\ep}(1)$, and index sets $I,J$, both of size $n-O_{C}(n^\ep)$, with $I\cap I_0 = \emptyset, J\cap J_0 =\emptyset$, and there exist integers $k,l, k_{ii_0},l_{jj_0}, i_0\in I_0, j_0\in J_0, i\in I, j\in J$, all of size bounded by $n^{O_{C,\ep}(1)}$, such that the following hold for all $i\in I$:

\begin{itemize}
 
\item for any $j\in J$,

$$a_{ij} = \frac{a_{ij}'}{kl} - \frac{\sum_{i_0\in I_0} k_{ii_0} a_{i_0j}}{k} - \frac{\sum_{j_0\in J_0}l_{j_0j} a_{ij_0}}{l};$$

\vskip .1in

\item all but $O_{C}(n^\ep)$ entries $a'_{ij}$ belong to a proper symmetric GAP $Q_i$ depending on $i$, which has rank $O_{C,\ep}(1)$ and size $n^{O_{C,\ep}(1)}$.

\end{itemize}

\end{theorem}

Although Theorem \ref{theorem:ILO:bilinear:differentGAP} is enough for our later application, it does not yet reflect the examples given, namely the additive structures $Q_i$ corresponding to each row can be totally different. In the next theorem we show that these GAPs can be unified into a structure similar to a GAP.

\begin{theorem}[Inverse Littlewood-Offord theorem for bilinear forms, common structure]\label{theorem:ILO:bilinear:commonGAP}
Let $\ep<1, C$ be positive constants. Assume that

$$\rho_b(A)\ge n^{-C}.$$ 

Then there exist index sets $I_0,J_0$, both of size $O_{C,\ep}(1)$, and index sets $I,J$, both of size $n-O_{C}(n^\ep)$, with $I\cap I_0 = \emptyset, J\cap J_0 =\emptyset$, and there exist integers $k,l, k_{ii_0},l_{jj_0}, i_0\in I_0, j_0\in J_0, i\in I, j\in J$, all of size bounded by $n^{O_{C,\ep}(1)}$, such that for all $i\in I$ the following hold:

\begin{itemize}
 
\item for any $j\in J$,

$$a_{ij} = \frac{a_{ij}'}{kl} - \frac{\sum_{i_0\in I_0} k_{ii_0} a_{i_0j}}{k} - \frac{\sum_{j_0\in J_0}l_{j_0j} a_{ij_0}}{l};$$

\vskip .1in

\item all but $O_{C}(n^\ep)$ entries $a'_{ij}$  belong to a set $Q$ (independent of $i$) of the form 

$$Q= \{\sum_{h=1}^{O_{C,\ep}(1)} (p_h/q_h)\cdot g_h; p_h,q_h\in \Z, |p_h|,|q_h|=n^{O_{C,\ep}(1)}\}.$$

\end{itemize}

\end{theorem}

Our proof of Theorem \ref{theorem:ILO:bilinear:differentGAP} and \ref{theorem:ILO:bilinear:commonGAP} can be extended (rather automatically) to other Bernoulli distributions.

\begin{theorem}\label{theorem:ILO:bilinear:general} Let $0<\mu \le 1$ be a constant. Then the conclusions of Theorem \ref{theorem:ILO:bilinear:differentGAP} and Theorem \ref{theorem:ILO:bilinear:commonGAP} also hold if we assume that $\rho_b^{(\mu)}(A)\ge n^{-C}$.  
\end{theorem}

\vskip .1in

\begin{remark}
The inverse Littlewood-Offord problem for bilinear forms was also studied in \cite{C}, but only for the case $\rho_b(A)\ge n^{-1+o(1)}$.
\end{remark}

\section{The inverse  Littlewood-Offord problem for quadratic forms}\label{section:quadraticILO}

Let $x_i$ be iid Bernoulli random variables, let $A=(a_{ij})$ be an $n\times n$ symmetric matrix of real entries. We define the {\it quadratic concentration probability} of $A$ by

$$\rho_q(A):= \sup_{a \in \R}\P(\sum_{i,j} a_{ij}x_ix_j=a).$$

More general, if $x_i$ are iid copies of $\eta^\mu$, then the {\it weighted quadratic concentration probability} of $A$ is defined by

$$\rho_q^{(\mu)}(A):= \sup_{a \in \R}\P(\sum_{i,j} a_{ij}x_ix_j=a).$$

It was shown in \cite{C,CTV}, as an application of Theorem \ref{theorem:bilinear:LOinequality},  that  

\begin{theorem}[Quadratic Littlewood-Offord inequality]\label{theorem:quadratic:LOinequality} Suppose that there are $\Theta(n)$ indices $i$ such that for each $i$ there are $\Theta(n)$ indices $j$ such that $a_{ij}\neq 0$. Then

$$\rho_q(A)\le n^{-1/2+o(1)}.$$ 
\end{theorem}

The bound $n^{-1/2+o(1)}$ is almost best possible, as demonstrated by the quadratic form $\sum_{ij}x_ix_j$. 

A more general version of Theorem \ref{theorem:quadratic:LOinequality} also appeared in the mentioned paper of Rosi\'nski and Samorodnitsky. 

\begin{theorem} \cite[Theorem 3.1]{RS}  Let $\Phi$ be a symmetric convex measurable set in a vector space $V$, and $a_i\in V$. Assume that there are $\Theta(n)$ indices $i$ such that for each $i$ there are $\Theta(n)$ indices $j$ such that $a_{ij}\notin \Phi$. Then we have

$$\sup_{a\in V}\P(\sum_{i,j} a_{ij}x_ix_j \in a+\Phi) = O(n^{-1/16}).$$
\end{theorem}

Motivated by the inverse Littlewood-Offord results for linear forms and bilinear forms, we would like to characterize those $A$ which have large quadratic concentration probability. 

We first consider a few examples of $A$ when $\rho_q(A)$ is large, based on the examples given in the previous sections.

\begin{example}[Additive structure implies large concentration probability]\label{example:quadratic:1} Let $Q$ be a proper symmetric GAP of rank $r=O(1)$ and of size $n^{O(1)}$. Assume that $a_{ij} \in Q$, then for any $x_i\in \{\pm 1\}$

$$\sum_{i,j}a_{ij}x_ix_j\in n^2Q.$$ 

Thus, by the pigeon-hole principle, 

$$\rho_q(A)\ge n^{-2r}|Q|^{-1} =n^{-O(1)}.$$ 
  
\end{example}

Similar to Example \ref{example:bilinear:2}, our next example shows that if the $a_{ij}$ are separable, then $\rho_q(A)$ is large.

\begin{example}[Algebraic structure implies large concentration probability]\label{example:quadratic:2}
Assume that 

$$a_{ij}= k_ib_j+k_jb_i$$ 

where $k_i\in \Z, |k_i|=n^{O(1)}$ and such that $\P_x(\sum_i k_ix_i= 0)=n^{-O(1)}$. 

Then we have 

$$\P(\sum_{i,j}a_{ij}x_ix_j =0) =\P(\sum_{i}k_ix_i  \sum_jb_jx_j =0)=n^{-O(1)}.$$

\end{example}

In our last example, we show that a combination of both structures also implies high quadratic concentration probability. 

\begin{example}[Structure implies large concentration probability]\label{example:quadratic:3}
Assume that $a_{ij}=a_{ij}' +a_{ij}''$, where $a_{ij}'\in Q$, a proper symmetric GAP of rank $O(1)$ and size $n^{O(1)}$, and 

$$a_{ij}''= k_{i1}b_{1j}+k_{j1}b_{1i}+\dots+k_{ir}b_{rj}+k_{jr}b_{ri},$$ 

where $b_{1i},\dots, b_{ri}$ are arbitrary and $k_{i1},\dots,k_{ir}$ are integers bounded by $n^{O(1)}$, and $r=O(1)$ such that

$$\P_{x}\left(\sum_i k_{i1}x_i=0,\dots,\sum_i k_{ir}x_i=0\right)=n^{-O(1)}.$$ 

Then we have 

$$\sum_{i,j} a_{ij}x_ix_j = \sum_{i,j}a_{i,j}'x_ix_j + (\sum_{i} k_{i1}x_i)(\sum_{j} b_{1j}x_j)+\dots + (\sum_{i} k_{ir}x_i)(\sum_{i} b_{rj}x_j).$$

Thus,

$$\P_x(\sum_{i,j}a_{ij}x_ix_j \in n^2Q)=n^{-O(1)}.$$

It then follows, by the pigeon-hole principle, that $\rho_q(A)=n^{-O(1)}$.
\end{example}

Next we state our main result which asserts that the examples above are essentially the only ones that have high quadratic concentration probability.

\begin{theorem}[Inverse Littlewood-Offord theorem for quadratic forms]\label{theorem:ILO:quadratic:differentGAP}
Let $\ep<1, C$ be positive constants. Assume that 

$$\rho_q(A)\ge n^{-C}.$$ 

Then there exist index sets $I_0$ and $I$ of size $O_{C,\ep}(1)$ and $n-O_C(n^\ep)$ respectively, and $I\cap I_0=\emptyset$, and there exist integers $k, k_{ii_0}\in \Z, i_0 \in I_0, i\in I$, all bounded by $n^{O_{C,\ep}(1)}$, such that the following hold for all $i\in I$: 

\begin{itemize}

\item for any $j\in I$,  

$$a_{ij} = a_{ij}'/k^2 - k\sum_{i_0\in I_0} k_{ii_0} a_{i_0j}/k^2 - k\sum_{i_0\in I_0}k_{ji_0} a_{i_0i}/k^2;$$

\item all but $O_C(n^\ep)$ entries $a'_{ij}$ belong to a proper symmetric GAP $Q_i$ depending on $i$, which has rank $O_{C,\ep}(1)$ and size $n^{O_{C,\ep}(1)}$.
\end{itemize}

\end{theorem}

Similar to Theorem \ref{theorem:ILO:bilinear:commonGAP}, we show that the structures $Q_i$ from Theorem \ref{theorem:ILO:quadratic:differentGAP} can be unified into a structure similar to a GAP.

\begin{theorem}[Inverse Littlewood-Offord theorem for quadratic forms, common structure]\label{theorem:ILO:quadratic:commonGAP} Let $\ep<1, C$ be positive constants. Assume that

$$\rho_q(A)\ge n^{-C}.$$ 

Then there exist index sets $I_0, I$ of size $O_{C,\ep}(1)$ and $n-O_C(n^\ep)$ respectively, with $I\cap I_0 = \emptyset$, and there exist integers $k,k_{ii_0},i_0\in I_0,i\in I$, all of size bounded by $n^{O_{C,\ep}(1)}$, such that for all $i\in I$ the following hold:

\begin{itemize}
 
\item for any $j\in I$,

$$a_{ij} = a_{ij}'/k^2 - k\sum_{i_0\in I_0} k_{ii_0} a_{i_0j}/k^2 - k\sum_{i_0\in I_0}k_{ji_0} a_{i_0i}/k^2;$$

\item all but $O_C(n^\ep)$ entries $a_{ij}'$ belong to a set $Q$ (independent of $i$) of the form 

$$Q= \{\sum_{h=1}^{O_C(1)} (p_h/q_h)\cdot g_h; p_h,q_h\in \Z, |p_h|,|q_h|=n^{O_{C,\ep}(1)}\}.$$

\end{itemize}
\end{theorem}

\vskip .1in

\begin{remark} The conclusions of Theorem \ref{theorem:ILO:quadratic:differentGAP} and Theorem \ref{theorem:ILO:quadratic:commonGAP} also hold if we assume that 

$$\rho_q^{(\mu)}(A)\ge n^{-C}.$$ 

We invite the reader to prove this result using the approach presented in Section \ref{section:proof:quadraticILO}.  
\end{remark}

\vskip .1in

\begin{remark}
The inverse Littlewood-Offord problem for quadratic forms was also studied in \cite{C}, but only in the case $\rho_q(A)\ge n^{-1/2+o(1)}$.
\end{remark}

\section{A rank reduction argument and the full rank assumption}\label{appendix:fullrank}
This section, which can be read independently of the rest of this paper, provides a technical lemma we will need for later sections. Informally, it says that if we can find a proper symmetric GAP that contains a given set (in the spirit of Sections \ref{section:linearILO}, \ref{section:bilinearILO} and \ref{section:quadraticILO}), then we can assume this containment is non-degenerate. More details follow.  

Assume that $P=\{m_1g_1+\dots+m_rg_r | -M_i\le m_i \le M_i\}$ is a proper symmetric GAP, which contains a set $U=\{u_1,\dots. u_n\}$. 

We consider $P$ together with the map $\Phi: P \rightarrow \R^r$ which maps $m_1g_1+\dots+m_rg_r$ to $(m_1,\dots,m_r)$. Because $P$ is proper, this map is bijective. 

We know that $P$ contains $U$, but we do not know yet that $U$ is non-degenerate in $P$ in the sense that the set $\Phi(U)$ has full rank in $\R^{r}$. In the later case, we say $U$ {\it spans} P.

\begin{theorem}\label{theorem:fullrank}
Assume that $U$ is a subset of a proper symmetric GAP $P$ of size $r$, then there exists a proper symmetric GAP $Q$ that contains $U$ such that the following hold.

\begin{itemize}
\item $\rank(Q)\le r$ and $|Q|\le O_r(1)|P|$;

\vskip .1in

\item $U$ spans $Q$, that is, $\phi(U)$ has full rank in $\R^{\rank(Q)}$.
\end{itemize}

\end{theorem}

To prove Theorem \ref{theorem:fullrank}, we will rely on the following lemma.

\begin{lemma}[Progressions lie inside proper progressions]\cite[Chapter 3.]{TVbook}\label{lemma:embeding}
There is an absolute constant $C$ such that the following holds. Let $P$ be a GAP of rank $r$ in $\R$. Then there is a symmetric proper GAP $Q$ of rank at most $r$ containng $P$ and 

$$|Q|\le r^{Cr^3}|P|.$$ 

\end{lemma}

\begin{proof} (of Theorem \ref{theorem:fullrank}) We shall mainly follow \cite[Section 8]{TVsing}.

Suppose that $\Phi(U)$ does not have full rank, then it is contained in a hyperplane of $\R^r$. In other words, there exist integers $\alpha_1,\dots,\alpha_r$ whose common divisor is one and $\alpha_1m_1+\dots + \alpha_r m_r=0$ for all $(m_1,\dots,m_r)\in \Phi(U)$.

Without loss of generality, we assume that $\alpha_r \neq 0$. We select $w$ so that $g_r=\alpha_r w$, and consider $P'$ be the GAP generated by $g_i':=g_i-\alpha_iw$ for $1\le i \le r-1$. The new symmetric GAP $P'$ will continue to contain $U$, because we have 

\begin{align*}
m_1g_1'+\dots +m_{r-1}g_{r-1}' &= m_1g_1+\dots+m_rg_r - w(\alpha_1m_1+\dots+\alpha_rg_r)\\
&= m_1g_1+\dots+m_rg_r
\end{align*}

for all $(m_1,\dots,m_r)\in \Phi(U)$. 

Also, note that the volume of $P'$ is $2^{r-1}M_1\dots M_{r-1}$, which is less than the volume of $P$.

We next use Lemma \ref{lemma:embeding} to guarantee that $P'$ is symmetric and proper without increasing the rank. 
 
Iterate the process if needed. Because we obtain a new proper symmetric GAP whose rank strictly decreases each step, the process must terminate after at most $r$ steps.
 
\end{proof}

\section{proof of Theorem \ref{theorem:ILO:bilinear:differentGAP} , Theorem \ref{theorem:ILO:bilinear:commonGAP}, and Theorem \ref{theorem:ILO:bilinear:general} }\label{section:proof:bilinearILO}

We begin by applying Theorem \ref{theorem:ILO:general}.

\begin{lemma}\label{lemma:roworthogonal} 
Let $\ep<1$, $0< \mu \le 1$, and $C$ be positive constants. Assume that $\rho_b^{(\mu)}(A)\ge n^{-C}$. Then the following holds with probability at least $\frac{3}{4}{n^{-C}}$ with respect to $y=(y_1,\dots,y_n)$. There exist a proper symmetric GAP $Q_y$ of rank $O_{C,\ep,\mu}(1)$ and size $O_{C,\ep,\mu}(1/\rho_b^{(\mu)})$ and a set $I_y$ of $n-n^\ep$ indices such that for each $i\in I_y$ we have

$$\langle \row_i, y \rangle \in Q_y.$$
\end{lemma}

\begin{proof}(of Lemma \ref{lemma:roworthogonal})
For short we write

$$\sum_{i,j} a_{ij}x_iy_j = \sum_{i=1}^n x_i \langle \row_i,y \rangle .$$

We say that a vector $y=(y_1,\dots,y_n)$ is {\it good} if

$$\P_x(\sum_{i=1}^n x_i\langle \row_i,y \rangle =a)\ge \rho_b^{(\mu)}/4.$$
 
We call $y$ {\it bad} otherwise.

First, we estimate the probability $p$ of a randomly chosen vector $y=(y_1,\dots,y_n)$ being bad by an averaging method.

\begin{align*}
\P_y \P_x \sum_{i=1}^n \langle \row_i,y \rangle &=\rho_b^{(\mu)}\\
p \rho_b^{(\mu)}/4 + 1-p &\ge \rho_b^{(\mu)}.\\
(1-\rho_b^{(\mu)})/(1-\rho_b^{(\mu)}/4) &\ge p.
\end{align*}

Thus, the probability of a randomly chosen vector being good is at least 

$$1-p \ge (3\rho_b^{(\mu)}/4)/(1-\rho_b^{(\mu)}/4) \ge 3\rho_b^{(\mu)}/4.$$

Next, we consider a good vector $y\in G$. By definition, we have

$$\P_x (\sum_{i=1}^n x_i \langle \row_i,y \rangle =a ) \ge \rho_b^{(\mu)}/4 .$$

A direct application of Theorem \ref{theorem:ILO:general} to the sequence $\langle \row_i,y \rangle$, $i=1,\dots,n$ yields the desired result.
\end{proof}

By Theorem \ref{theorem:fullrank}, we may assume that the $\langle \row_i,y \rangle$ span $Q_y$. From now on we fix such a $Q_y$ for each $y$. 

Let $G$ be the collection of good vectors. Thus, 

\begin{equation}
\P_y(y\in G)\ge 3\rho_b^{(\mu)}/4.
\end{equation}

Next, for each $y\in G$, we choose from $I_y$ $s$ indices $i_{y_1},\dots,i_{y_s}$ such that $\langle \row_{i_{y_j}},y\rangle $ span $Q_y$, where $s$ is the rank of $Q_y$. We note that $s=O_{C,\ep,\mu}(1)$ for all $s$. 

Consider the tuples $(i_{y_1},\dots,i_{y_s})$ for all $y\in G$. Because there are $\sum_{s} O_{C,\ep,\mu}(n^s) = n^{O_{C,\ep,\mu}(1)}$ possibilities these tuples can take, there exists a tuple, say $(1,\dots,r)$ (by rearranging the rows of $A$ if needed, we may assume so), such that $(i_{y_1},\dots,i_{y_s})=(1,\dots,r)$ for all $y\in G'$, a subset of $G$ satisfying 

\begin{equation}
\P_y(y\in G')\ge \P_y(y\in G)/n^{O_{C,\ep,\mu}(1)} =\rho_b^{(\mu)}(A)/n^{O_{C,\ep,\mu}(1)}.
\end{equation}

For each $1\le i\le r$, we express $\langle \row_i,y \rangle$ in terms of the generators of $Q_y$ for each $y\in G'$, 

$$\langle \row_{i},y\rangle = c_{i1}(y)g_{1}(y)+\dots + c_{ir}(y)g_{r}(y),$$ 

where $c_{i1}(y),\dots c_{ir}(y)$ are integers bounded by $n^{O_{C,\ep,\mu}(1)}$, and $g_{i}(y)$ are the generators of $Q_y$.

We will show that there are many $y$ that correspond to the same coefficients $c_{ij}$. 

Consider the collection of the coefficient-tuples $\Big(\big(c_{11}(y),\dots,c_{1r}(y)\big);\dots; \big(c_{r1}(y),\dots c_{rr}(y)\big)\Big)$ for all $y\in G'$. Because the number of possibilities these tuples can take is at most

$$(n^{O_{C,\ep,\mu}(1)})^{r^2} =n^{O_{C,\ep,\mu}(1)}.$$

There exists a coefficient-tuple, say  $\Big((c_{11},\dots,c_{1r}),\dots, (c_{r1},\dots c_{rr})\Big)$, such that  

$$\Big(\big(c_{11}(y),\dots,c_{1r}(y)\big);\dots; \big(c_{r1}(y),\dots c_{rr}(y)\big)\Big) =\Big((c_{11},\dots,c_{1r}),\dots, (c_{r1},\dots c_{rr})\Big)$$  

for all $y\in G''$, a subset of $G'$ satisfying 

\begin{equation}
\P_y(y\in G'')\ge \P_y(y\in G')/n^{O_{C,\ep,\mu}(1)} \ge \rho_b^{\mu}(A)/n^{O_{C,\ep,\mu}(1)}.
\end{equation}

In summary, there exist $r$ tuples  $(c_{11},\dots,c_{1r}),\dots, (c_{r1},\dots c_{rr})$, whose components are integers bounded by $n^{O_{C,\ep,\mu}(1)}$, such that the following hold for all $y\in G''$.

\begin{itemize}

\item $\langle \row_{i},y\rangle = c_{i1}g_{1}(y)+\dots + c_{jr}g_{r}(y)$, for $i=1,\dots,r$.

\vskip .1in

\item The vectors  $(c_{11},\dots,c_{1r}),\dots, (c_{r1},\dots c_{rr})$ span $\Z^{\rank(Q_y)}$.

\end{itemize}

Next, because $|I_y|\ge n-n^\ep$ for each $y\in G''$, there is a set $I$ of size $n-3n^\ep$ such that $I \cap \{1,\dots,r\} =\emptyset$ and for each $i\in I$ we have 

\begin{equation}\label{eqn:optional}
\P_y(i\in I_y, y\in G'') \ge \P_y(y\in G'')/2.
\end{equation}  

Indeed, let $I'$ be the set of $i$ satisfying \eqref{eqn:optional}. Then, as 

$$\sum_{i}\sum_{y\in G'', i\in I_y} 1 = \sum_{y\in G''}\sum_{i\in I_y} 1 \ge (n-n^\ep)|G''|,$$ 

we have $\sum_{i\in I'}|G''|+\sum_{i\notin I'} |G''|/2 \ge (n-n^\ep)|G''|$. Hence, 

$$|I'||G''|+ (n-|I'|)|G''|/2 \ge (n-n^\ep)|G''|,$$ 

from which we deduce that $|I'|\ge n-2n^\ep$. To obtain $I$ we just remove the elements of $\{1,\dots,r\}$ from $I'$.

Now fix an arbitrary row $\row$ of index from $I$. We concentrate on those $y\in G''$ where the index of $\row$ belongs to $I_y$. 

Because $\langle \row,y\rangle \in Q_y$, we can write 

$$\langle \row, y \rangle = c_{1}(y)g_{1}(y)+\dots c_{r}(y)g_{r}(y)$$ 

where $c_{i}(y)$ are integers bounded by $n^{O_{C,\ep,\mu}(1)}$.

For short, we denote the vector $(c_{i1},\dots,c_{ir})$ by $\mathbf{v}_i$ for each $i$. We will also denote the vector $(c_{1}(y),\dots c_{r}(y))$ by $\mathbf{v}_{\row,y}$. 

Because $Q_i$ is spanned by $\langle \row_{1},y \rangle ,\dots, \langle \row_{r},y \rangle$, we have $k=\det(\mathbf{v}_1,\dots \mathbf{v}_r)\neq 0$, and that 

$$k \langle \row,y \rangle + \det(\mathbf{v}_{\row,y},\mathbf{v}_2,\dots,\mathbf{v}_r)\langle \row_{1},y \rangle +\dots + \det(\mathbf{v}_{\row,y},\mathbf{v}_1,\dots,\mathbf{v}_{r-1})\langle \row_{r},y \rangle =0.$$ 

Next, because each coefficient of the identity above is bounded by $n^{O_{C,\ep,\mu}(1)}$, there exists a subset $G_{\row}''$ of $G''$ such that all $y\in G_{\row}''$ correspond to the same identity, and

\begin{equation}
\P_y(y\in G_{\row}'') \ge (\P_y(y\in G'')/2)/(n^{O_{C,\ep,\mu}(1)})^r = \rho_b^{(\mu)}/n^{O_{C,\ep,\mu}(1)}.
\end{equation}

In other words, there exist integers $k_1,\dots,k_r$, all bounded by $n^{O_{C,\ep,\mu}(1)}$, such that 

$$k \langle \row,y \rangle + k_1 \langle \row_{1},y \rangle + \dots + k_r \langle \row_{r},y \rangle=0$$

for all $y\in G_{\row}''$. 

Note that $k$ is independent of $\row$ and $y$. We thus conclude below.

\begin{lemma}[The rows are mutually orthogonal to many $\{-1,0,1\}$ vectors]\label{lemma:rowtrullyorthogonal}
Let $i$ be any index of $I$. Then there are numbers $k_{i1},\dots, k_{ir}\in \Z$, all bounded by $n^{O_{C,\ep,\mu}(1)}$, such that 

$$\P_y\big(k \langle \row_i,y \rangle + \sum_{j=1}^{r} k_{ij} \langle \row_{j},y\rangle =0\big)=\rho_b^{(\mu)}/n^{O_{C,\ep,\mu}(1)}.$$
\end{lemma}

Putting Lemma \ref{lemma:roworthogonal}, Lemma \ref{lemma:rowtrullyorthogonal}, and Theorem \ref{theorem:ILO:general} together, we obtain the following result.

\begin{theorem}[Refined row relation]\label{theorem:rowrelation} Let $0<\ep \le 1$, $0<\mu \le 1$, and $C$ be positive constants.
Assume that $\rho_b^{(\mu)}(A)\ge n^{-C}$. Then there exist a set $I_0$ of size $O_{C,\ep,\mu}(1)$, a set $I$ of size $n-3n^\ep$ with $I\cap I_0=\emptyset$, and there exists a nonzero integer $k$ of size $n^{O_{C,\ep,\mu}(1)}$ such that the following holds for all $i\in I$: there exists a proper symmetric GAP $Q_i$ of rank $O_{C,\ep,\mu}(1)$ and size $n^{O_{C,\ep,\mu}(1)}$, an index set $J_i$ of size $n-n^\ep$, and integers $k_{ii_0},i_0\in I_0$, all bounded by $n^{O_{C,\ep,\mu}(1)}$, such that the following holds for all $j\in J_i$

$$\sum_{i_0\in I_0} k_{ii_0} a_{i_0j} + k a_{ij}\in Q_i.$$

\end{theorem}

Because the role of rows and columns of $A$ can be swapped, we obtain a similar conclusion for the columns of $A$.

\begin{theorem}[Refined column relation]\label{theorem:columnrelation} Let $0<\ep\le 1, 0<\mu\le 1$, and $C$ be positive constants.
Assume that $\rho_b^{(\mu)}(A)\ge n^{-C}$.  
Then there exist a set $J_0$ of size $O_{C,\ep,\mu}(1)$, a set $J$ of size $n-3n^\ep$ with $J\cap J_0=\emptyset$, and there exists a nonzero integer $l$ of size $n^{O_{C,\ep,\mu}(1)}$ such that the following holds for all $j\in J$: there exists a proper symmetric GAP $P_j$ of rank $O_{C,\ep,\mu}(1)$ and size $n^{O_{C,\ep,\mu}(1)}$, an index set $I_j$ of size $n-n^\ep$, and integers $l_{j_0j},j_0\in J_0$, all bounded by $n^{O_{C,\ep,\mu}(1)}$, such that the following holds for all $i\in I_j$

$$\sum_{j_0\in J_0} l_{j_0j} a_{ij_0}+ l a_{ij}\in P_j.$$

\end{theorem}

Next we introduce the following two matrices.

\begin{definition}[Row matrix] $L$ is an $n$ by $n$ matrix, whose $i$-th row, where $i\in I$, is defined by 

\begin{equation}\label{eqn:R}
\row_i(L)(j):= 
\begin{cases}
k_{ij}, & \text{ if $j\in I_0$;}\\ 
k, & \text{ if $j=i$;}\\
0, & \text{ otherwise.}
\end{cases}
\end{equation}

The other entries of $L$ are zero, except the diagonal terms which are set to be 1.

\end{definition}

\begin{definition}[Column matrix] $R$ is an $n$ by $n$ matrix, whose $j$-th column, where $j\in J$, is defined by 

\begin{equation}\label{eqn:C}
\col_j(R)(i):=
\begin{cases}
l_{ij},& \text{ if $i\in J_0$;}\\
l, & \text{ if $i=j$;}\\ 
0, & \text{ otherwise.} 
\end{cases}
\end{equation}

The other entries of $R$ are zero, except the diagonal terms which are set to be 1.

\end{definition}

\begin{remark} For each $i\in I$, the non-singular matrix $L$ acts on the left of $A$ by rescaling $\row_i(A)$ by a factor of $k$, modulo $\sum_{i_0\in I_0} k_{ii_0}\row_{i_0}$. For each $j\in J$, the non-singular matrix $R$ acts on the right of $A$ by rescaling $\col_j(A)$ by a factor of $l$, modulo $\sum_{j_0\in J_0} l_{j_0j} \col_{j_0}$.
\end{remark}

Define 

$$A':=L A R.$$ 

First, consider the matrix $AR$. By definition, $(AR)_{ij}\in P_j$ for all $i\in I_j$, where $j\in J$. By adding a constant number of generators to $P_j$ we may assume that $(AR)_{ij} \in P_j$, where $i\in I_0$. 

Next, consider the  matrix $A'=LAR$. Suppose that $j\in J$, then we have 

$$(LAR)_{ij}= k (AR)_{ij} + \sum_{i_0\in I_0}k_{ii_0} (AR)_{i_0j}.$$ 

Because $k,k_{ii_0}=n^{O_{C,\ep,\mu}(1)}$, it thus follows that $(LAR)_{ij} \in n^{O_{C,\ep,\mu}(1)}\cdot P_j$ whenever $i\in I_j\cap I$. To avoid notational complication, we keep the same notation $P_j$ for this new proper symmetric GAP (which is still of rank $O_{C,\ep,\mu}(1)$ and size $n^{O_{C,\ep,\mu}(1)}$, with possibly worse constants).
 
We have just shown that for each $j\in J$ there exists a proper symmetric GAP $P_j$ of rank $O_{C,\ep,\mu}(1)$ and size $n^{O_{C,\ep,\mu}(1)}$ such that all but $n^\ep$ coordinates of the $j$-th column of $A'$ belong to $P_j$.

Similarly, by viewing $LAR$ as $(LA)R$, we infer that for each $i\in I$, there exists a proper symmetric GAP $Q_i$ of rank $O_{C,\ep,\mu}(1)$ and size $n^{O_{C,\ep,\mu}(1)}$ such that all but $n^\ep$ coordinates of the $i$-th row of $A'$ belong to $Q_i$.

Putting everything together, we obtain the following result.

\begin{theorem}[Matrix relation]\label{theorem:matrixrelation} Let $0<\ep \le 1$, $0<\mu \le 1$, and $C$ be positive constants. 
Assume that $\rho_b^{(\mu)}(A)\ge n^{-C}$. Then there exist index sets $I_0,J_0$, both of size $O_{C,\ep}(1)$, and index sets $I,J$, both of size $n-3n^\ep$, with $I\cap I_0 = \emptyset, J\cap J_0 =\emptyset$, such that the following holds. There exist two matrices $L,R$ defined by \eqref{eqn:R} and \eqref{eqn:C} respectively such that the matrix $A'=LAR$ possess the following properties.

\begin{itemize}

\item For each $i\in I$, there exist  a subset $\row_i' \subset \row_i(A')$ of size $n-n^\ep$ and a proper symmetric GAP $Q_i$ of rank $O_{C,\ep,\mu}(1)$ and size $n^{O_{C,\ep,\mu}(1)}$ such that $\row_i' \subset Q_i$.

\vskip .1in

\item For each $j\in J$, there exist a subset $\col_j' \subset \col_j(A')$ of size $n-n^\ep$ and a proper symmetric GAP $P_j$ of rank $O_{C,\ep,\mu}(1)$ and size $n^{O_{C,\ep,\mu}(1)}$ such that $\col_j' \subset P_j$.

\end{itemize}
\end{theorem}

We now deduce Theorem \ref{theorem:ILO:bilinear:differentGAP}. Assume that $i\in I$ and $j\in J$. We then have  

\begin{align*}
a_{ij}' = &kl a_{ij} + \sum_{i_0\in I_0,j\in J_0}k_{ii_0}a_{i_0j_0}l_{j_0j} \\
&+ l\sum_{i_0\in I_0} k_{ii_0} a_{i_0j} + k\sum_{j_0\in J_0}l_{j_0j} a_{ij_0}.
\end{align*}

This identity implies

\begin{eqnarray}\label{eqn:aa'}
a_{ij} &=& \frac{a_{ij}'}{kl} - \sum_{i_0\in I_0,j_0\in J_0} \frac{k_{ii_0} l_{j_0j} a_{i_0,j_0}}{kl} \nonumber \\
      &-& \sum_{i_0\in I_0} \frac{k_{ii_0} a_{i_0j}}{k} - \sum_{j_0\in J_0}\frac{l_{j_0j} a_{ij_0}}{l}.
\end{eqnarray}

To complete the proof of Theorem \ref{theorem:ILO:bilinear:differentGAP} we just need to add $a_{i_0j_0}$ to the set of the generators of $Q_i$.

To finish the proof of Theorem \ref{theorem:ILO:bilinear:commonGAP}, it is enough to show that the proper symmetric GAPs from Theorem \ref{theorem:matrixrelation} can be unified.

\begin{lemma}\label{lemma:commonGAP}
Assume that for each $i\in I$, there exist  a subset $\row_i' \subset \row_i$ of size $n-n^\ep$ and a proper symmetric GAP $Q_i$ of rank $O_{C,\ep,\mu}(1)$ and size $n^{O_{C,\ep,\mu}(1)}$ such that $\row_i' \subset Q_i$, and for each $j\in J$, there exist a subset $\col_j' \subset \col_j$ of size $n-n^\ep$ and a proper symmetric GAP $P_j$ of rank $O_{C,\ep,\mu}(1)$ and size $n^{O_{C,\ep,\mu}(1)}$ such that $\col_j' \subset P_j$. Then there exist a bounded number of generators $g_1,\dots,g_s$, where $s=O_{C,\ep,\mu}(1)$, such that the set $\{\sum_{h=1}^s (p_h/q_h) g_h, |p_h|,|q_h|=n^{O_{C,\ep,\mu}(1)}\}$ contains all but at most $\ep n$ entries of all but at most $\ep n$ rows of $A$.
\end{lemma}

It is clear that Theorem \ref{theorem:ILO:bilinear:commonGAP} follows from Lemma \ref{lemma:commonGAP}. It thus remains to verify this lemma.   

\begin{proof}(of Lemma \ref{lemma:commonGAP}) Throughout the proof, if not specified, all the rows and columns will have index in $I$ and $J$ respectively. We assume that all the proper GAPs has rank at most $r=O_{C,\ep}(1)$. 

By throwing away at most $\ep n /2$ rows, we may assume that for each row $\row_i$ all but at most $n^\ep/2\ep$ indices $j$ satisfy $\row_i(j)\in \col_j' \subset P_j$. Let $\row_i'$ be the collection of these $\row_i(j)$ for each $i$. 

Set 

$$\delta = \ep/2r.$$

Consider an arbitrary $\row_i'$. It's components are combinations of the generators of $Q_i$. Thus we may view these elements of $\row_i'$ as vectors over $\Z^{\rank(Q_i)}$ (see Section \ref{appendix:fullrank}). We say that the elements of $\row_i'$ are {\it independent} if their defining vectors are independent. 

Next we will choose a subset $\row_i''$ of $\row_i'$ with the following properties.

\begin{enumerate}

\item $|\row_i''|\ge (1-\ep)n$.

\vskip .1in



\item Let $H_i$ be the subspace generated by the defining vectors of the components of $\row_i''$. Then any hyperplane of $H_i$ contains no more than $(1-\delta)|\row_i''|$ such defining vectors.

\end{enumerate}

We show that there must exist such $\row_i''$. 

Assume that $\row_i'$ does not have the above property. By definition of $\row_i'$, this means that (2) is not satisfied. We next pass to consider the set of at least $(1-\delta)|\row_i'|$ components that belong to a proper subspace. Assume that this set does not have the above properties either, we then keep iterating the process. Because the dimensions of the subspaces strictly decrease after each step, the process must terminate after at most $r$ steps. By definition, the subset $\row_i''$ obtained at the time of termination has the desired properties. 

Also, 

$$|\row_i''| \ge |\row_i'|- r\delta |\row_i'| = (1-\ep/2)(n-n^\ep/2\ep)\ge (1-\ep) n.$$

Now we will group some generators from the $P_j$'s to create a new set $S$.

We start with the first column $\col_{j_1}$ and put the generators of $P_{j_1}$ into $S$. Assume that  we already gathered the generators of $P_{j_1},\dots, P_{j_k}$ after $k$ steps.

To choose a $P_j$ for the next step, we consider the defining vectors of $\row_i''(j_1),\dots \row_i''(j_k)$ for each $i$. Let $\dim(\row_i''(j_1),\dots \row_i''(j_k))$ denote the dimension of the subspace generated by these vectors. 

By the definition of $\row_i''$, if  $\row_i''(j_1),\dots \row_i''(j_k)$ do not generate $H_i$ (in which case we say that $\row_i''$ is not {\it complete}), then there are at least $\delta(1-\ep)n \ge \delta n/2$ ways to choose $P_j$ so that  $\dim(\row_i''(j_1),\dots \row_i''(j_k),\row_i''(j)) =\dim(\row_i''(j_1),\dots \row_i''(j_k))+1$. In this case we say that there is an {\it increase in dimension} in $\row_i''$.  

Hence after some $k$ steps, if there are $\alpha n$ rows that are not complete, then, by the pigeon-hole principle, there is a choice for $P_j$ which results in an increase in dimension in at least $\alpha \delta n/2$ rows $\row_i''$.

Because the total of the dimensions is bounded by $rn$, there must be at least $(1-\ep)n$ rows that are complete after at most $2r/(\ep \delta) = 2r^2 \ep^2$ steps. Let $S$ be the collection of all the generators of $P_j$ considered until this step. The size $s$ of $S$ is then at most $2r^3/\ep^2$.

Consider a row $\row_i''$ that is complete. Assume that its elements are generated by $\row_i''(j_1),\dots, \row_i''(j_r)$, where $\row_i''(j_k)\in P_{j_k}$, a GAP whose generators belong $S$. Let $a$ be any element of $\row_i''$, and let $\mathbf{a}$ be its defining vector in $Q_i$, we then have 

$$a= \det\big(\mathbf{a},\row_i''(j_2),\dots,\row_i''(r)\big)\det\big(\row_i''(j_1),\dots,\row_i''(j_r) \big)^{-1}\cdot \row_i''(j_1)+\dots $$

$$+ \det\big(\row_i''(j_1),\dots,\row_i''(j_{r-1}),\mathbf{a} \big)\det\big(\row_i''(j_1),\dots,\row_i''(j_r)\big)^{-1}\cdot \row_i''(j_r).$$ 

Thus $a$ can be written in the form $\sum_{h=1}^s (p_h/q_h) \cdot g_h $, where $|p_h|, |q_h| =n^{O_{C,\ep}(1)}$. 

\end{proof}

\section{proof of Theorem \ref{theorem:ILO:quadratic:differentGAP} and Theorem \ref{theorem:ILO:quadratic:commonGAP}}\label{section:proof:quadraticILO}

In this section we will use the results from Section \ref{section:proof:bilinearILO} to prove Theorem  \ref{theorem:ILO:quadratic:differentGAP} and Theorem \ref{theorem:ILO:quadratic:commonGAP}.  

Let $U$ be a random subset of $\{1,\dots,n\}$, where $\P(i\in U)=1/2$ for each $i$. Let $A_U$ be a submatrix of $A$ defined by 

\[
A_U(ij)= 
\begin{cases}
a_{ij} & \text{ if either $i\in U, j\notin  U $ or $i\notin U, j\in U$},\\
0 & \text{otherwise.}
\end{cases}
\]

We first apply the following lemma.

\begin{lemma}[Concentration for bilinear forms controls concentration for quadratic forms]\label{lemma:decoupling}

$$\rho_q(A)^8 \le \P_{v,w} (\sum_{i,j}A_U(ij)v_iw_j =0),$$

where $v_i,w_j$ are iid copies of $\eta^{1/2}$.
\end{lemma}
 
\begin{proof}(of Lemma \ref{lemma:decoupling}) We first write  

$$\P_x(\sum_{i,j}a_{ij}x_ix_j=a) = \E_x \int_{0}^1 \exp\big(2\pi \sqrt{-1} (\sum_{i,j}a_{ij}x_ix_j-a)t\big)dt.$$ 

Hence, 

$$\P_x(\sum_{i,j}a_{ij}x_ix_j=a) \le \int_{0}^1 \big|\E_x \exp(2\pi \sqrt{-1} (\sum_{i,j}a_{ij}x_ix_j)t)dt\big|.$$

Next we consider $x$ as $(x_U,x_{\bar{U}})$, where $x_U, x_{\bar{U}}$ are the vectors corresponding to $i\in U$ and  $i\notin U$ respectively. By the Cauchy-Schwarz inequality 

\begin{align*}  
&\left(\int_{0}^1 \big|\E_x \exp(2\pi \sqrt{-1} (\sum_{i,j}a_{ij}x_ix_j)t)\big|dt\right)^4 \le \left(\int_{0}^1  \big|\E_x \exp(2\pi \sqrt{-1} (\sum_{i,j}a_{ij}x_ix_j)t)\big|^2dt\right)^2\\
&\le \left(\int_{0}^1 \E_{x_U}\big|\E_{x_{\bar{U}}}\exp(2\pi \sqrt{-1} (\sum_{i,j} a_{ij}x_ix_j)t)\big|^2 dt\right)^2 \\ 
&=\left( \int_{0}^1 \E_{x_U}\E_{x_{\bar{U}},x_{\bar{U}}'} \exp\Big(2\pi \sqrt{-1}\big( \sum_{i\in U,j\in \bar{U}}a_{ij}x_i(x_j-x_j')+\sum_{i\in \bar{U},j\in \bar{U}}a_{ij}(x_ix_j-x_i'x_j')\big)t\Big)dt\right)^2\\
&\le \int_{0}^1 \E_{x_{\bar{U}},x_{\bar{U}}'}\Big|\E_{x_{U}}\exp\Big(2\pi \sqrt{-1} \big(\sum_{i\in U,j\in \bar{U}}a_{ij}x_i(x_j-x_j')+\sum_{i\in \bar{U},j\in \bar{U}}a_{ij}(x_ix_j-x_i'x_j')\big)t\Big)\Big|^2 dt\\
&=\int_{0}^1 \E_{x_U,x_U',x_{\bar{U}},x_{\bar{U}}'} \exp\Big(2\pi \sqrt{-1} \big(\sum_{i\in U, j\in \bar{U}}a_{ij}(x_i-x_i')(x_j-x_j')\big)t\Big)dt.\\
&=\int_{0}^1\E_{y_{U},z_{\bar{U}}}\exp\Big(2\pi \sqrt{-1}(\sum_{i\in \bar{U},j\in U} a_{ij}y_iz_j)t\Big)dt,
\end{align*}

where $y_{U}=x_{U}-x_{U}'$ and $z_{\bar{U}}=x_{\bar{U}}-x_{\bar{U}}'$, whose entries are iid copies of $\eta^{1/2}$.

Thus we have 

$$\left(\int_{0}^1 \big|\E_x \exp(2\pi \sqrt{-1} (\sum_{i,j}a_{ij}x_ix_j)t)\big|dt\right)^8\le \left(\int_{0}^1\E_{y_U,z_{\bar{U}}}\exp\big(2\pi \sqrt{-1}(\sum_{i\in U,j\in \bar{U}} a_{ij}y_iz_j)t\big)dt\right)^2$$

$$\le \int_{0}^1\E_{y_U,z_{\bar{U}}, y_U', z_{\bar{U}}'}\exp\big(2\pi \sqrt{-1} (\sum_{i\in U,j\in \bar{U}}a_{ij}y_iz_j-\sum_{i\in U, j\in \bar{U}} a_{ij}y_i' z_j')t\big)dt.$$

Because $a_{ij}=a_{ji}$, we can write the last term as

\begin{align*}
&\int_{0}^1\E_{y_U,z_{\bar{U}}', y_U', z_{\bar{U}}}\exp\Big(2\pi \sqrt{-1} \big(\sum_{i\in U,j\in \bar{U}}a_{ij}y_iz_j+\sum_{j\in \bar{U}, i\in U} a_{ji}(-z_j')y_i'\big)t\Big)dt\\
&=\int_{0}^1 \E_{v,w}\exp\big(2\pi \sqrt{-1}( \sum_{i\in U, j\in \bar{U}}a_{ij}v_iw_j + \sum_{i\in \bar{U}, j\in U} a_{ij}v_iw_j)t\big)dt,
\end{align*}

where $v:=(y_U,-z_{\bar{U}}')$ and $w:=(y_U', z_{\bar{U}})$. 

To conclude the proof we observe that the entries of $v$ and $w$ are iid copies of $\eta^{1/2}$, and  

$$\int_{0}^1 \E_{v,w}\exp\Big(2\pi \sqrt{-1}\big( \sum_{i\in U j\in \bar{U}}a_{ij}v_iw_j + \sum_{i\in \bar{U}, j\in U} a_{ij}v_iw_j\big)t\Big)dt = \P_{v,w}(\sum_{i,j}A_U(ij)v_iw_j =0).$$

\end{proof}

Next, it follows from Lemma \ref{lemma:decoupling} that  

$$\P_{v,w} (\sum_{i,j}A_U(ij)v_iw_j =0)\ge n^{-8C}.$$

This inequality means that $\rho_q^{(1/2)}(A_U)\ge n^{-8C}$.  We now apply Lemma \ref{lemma:rowtrullyorthogonal}. 

\begin{lemma}\label{lemma:quadratic:row} There exist a set $I_0(U)$ of size $O_{C,\ep}(1)$ and a set $I(U)$ of size at least $n-n^\ep$ such that for any $i\in I$, there are integers $0\neq k(U)$ and $k_{ii_0}(U), i_0\in I_0(U)$, all bounded by $n^{O_{C,\ep}(1)}$, such that 

$$\P_y\big(\langle k(U)\row_{A_U}(i),y \rangle + \langle \sum_{i_0\in I_0} k_{ii_0}(U) \row_{A_U}(i_0),y \rangle = 0\big) = n^{-O_{C,\ep}(1)}.$$
\end{lemma}

Note that Lemma \ref{lemma:quadratic:row} holds for all $U$. We will try to obtain a similar conclusion for $A$.

As $I_0(U)\subset [n]^{O_{C,\ep}(1)}$ and $k(U)\le n$, there are only $n^{O_{C,\ep}(1)}$ possibilities that $(I_0(U),k(U))$ can take. Thus there exists a tuple $(I_0,k)$ such that 
$I_0(U)=I_0$ and $k(U)=k$ for $2^n/n^{O_{C,\ep}(1)}$ different $U$. Let us denote this set of $U$ by $\mathcal{U}$. Thus 

$$|\mathcal{U}|\ge 2^n/n^{O_{C,\ep}(1)}.$$

Next, let $I$ be the collection of $i$ which belong to at least $|\mathcal{U}|/2$ index sets $I_U$. Then we have
 
\begin{align*}                         
|I||\mathcal{U}| + (n-|I|)|\mathcal{U}|/2 & \ge (n-n^\ep )|\mathcal{U}|\\
|I| &\ge  n-2n^\ep.
\end{align*}

Fix an $i\in I$. Consider the tuples $(k_{ii_0}(U), i_0\in I_0)$ where $i\in I_U$. Because there are only $n^{O_{C,\ep}(1)}$ possibilities such tuples can take, there must be a tuple, say $(k_{ii_0}, i_0\in I_0)$, such that $(k_{ii_0}(U), i_0\in I_0)=(k_{ii_0}, i_0\in I_0)$ for at least $|\mathcal{U}|/2n^{O_{C,\ep}(1)}=2^n/n^{O_{C,\ep}(1)}$ sets $U$. 

Because $|I_0|=O_{C,\ep}(1)$, it is easy to see that there is a way to partition $I_0$ into $I_0' \cup I_0''$ such that there are $2^n/n^{O_{C,\ep}(1)}$ sets $U$ above satisfying that $I_0''\subset U$ and  $U\cap I_0'=\emptyset$. Let $\mathcal{U}_{I_0',I_0''}$ denote the collection of these $U$.

By passing to consider a subset of  $\mathcal{U}_{I_0',I_0''}$ if needed, we may assume that either $i\notin U$ or $i\in U$ for all $U\in  \mathcal{U}_{I_0',I_0''}$. Without loss of generality, we assume the first case that $i\notin U$. (The other case can be treated similarly).

Let $U\in \mathcal{U}_{I_0',I_0''}$ and $u=(u_1,\dots,u_n)$ be its characteristic vector, that is $u_j=1$ if $j\in U$, and $u_j=0$ otherwise. Then, by the definition of $A_U$, and because $I_0''\subset U$ and $I_0'\cap U=\emptyset$, for $i_0'\in I_0'$ and $i_0''\in I_0''$ we can respectively write 

$$\langle \row_{i_0'}(A_U),y \rangle = \sum_{j=1}^n a_{i_0'j}u_jy_j, \mbox{ and } \langle \row_{i_0''}(A_U),y \rangle = \sum_{j=1}^n a_{i_0''j}(1-u_j)y_j.$$

Also, because $i\notin U$, we have

$$\langle \row_{i}(A_U),y \rangle = \sum_{j=1}^n a_{ij}u_jy_j.$$

Thus, 

\begin{align*}
&\quad \langle k\row_i(A_U),y \rangle + \sum_{i_0\in I_0} \langle k_{ii_0} \row_{i_0}(A_U),y \rangle\\
& = \langle k\row_i(A_U),y \rangle +  \langle \sum_{i_0'\in I_0'} k_{ii_0'} \row_{i_0'}(A_U),y \rangle + \langle \sum_{i_0''\in I_0''} k_{ii_0''} \row_{i_0''}(A_U),y \rangle\\ 
&= \sum_{j=1}^n ka_{ij} u_jy_j + \sum_{j=1}^n \sum_{i_0'\in I_0'} k_{ii_0'} a_{i_0'j} u_jy_j +  \sum_{j=1}^n \sum_{i_0''\in I_0''} k_{ii_0''} a_{i_0''j} (1-u_j)y_j\\
&= \sum_{j=1}^n (ka_{ij} + \sum_{i_0'\in I_0'} k_{ii_0'} a_{i_0'j}- \sum_{i_0''\in I_0''} k_{ii_0''} a_{i_0''j} ) u_jy_j +  \sum_{j=1}^n \sum_{i_0''\in I_0''} k_{ii_0''} a_{i_0''j} y_j
\end{align*}

Next, by Lemma \ref{lemma:quadratic:row}, for each $U\in \mathcal{U}_{I_0',I_0''}$ we have  

$$\P_y\big (\langle k \row_i(A_U),y \rangle + \sum_{i_0\in I_0} \langle k_{ii_0} \row_{i_0}(A_U),y \rangle = 0\big)=n^{-O_{C,\ep}(1)}.$$ 

Also, note that 

$$|\mathcal{U}_{I_0',I_0''}|= 2^n/n^{O_{C,\ep}(1)}.$$ 

Hence, 

$$\E_y\E_U \big(k\langle \row_i(A_U),y \rangle + \sum_{i_0\in I_0} \langle k_{ii_0} \row_{i_0}(A_U),y \rangle =0\big) \ge n^{-O_{C,\ep}(1)}.$$

By applying the Cauchy-Schwarz inequality, we obtain 

\begin{align*}
n^{-O_{C,\ep}(1)}&\le \left(\E_y \E_U(k\langle \row_i(A_U),y \rangle + \sum_{i_0\in I_0} \langle k_{ii_0} \row_{i_0}(A_U),y \rangle =0)\right)^2 \\
&\le  \E_y \left(\E_U(k\langle \row_i(A_U),y \rangle + \sum_{i_0\in I_0} \langle k_{ii_0} \row_{i_0}(A_U),y \rangle =0\right)^2 \\
&= \E_y \left(\E_{u}(\sum_{j=1}^n (ka_{ij}+ \sum_{i_0'\in I_0'} k_{ii_0'} a_{i_0'j}-\sum_{i_0''\in I_0''} k_{ii_0''} a_{i_0''j}) u_jy_j+  \sum_{j=1}^n \sum_{i_0''\in I_0''} k_{ii_0''} a_{i_0''j} y_j= 0)\right)^2\\
&\le \E_y \E_{u,u'}\big(\sum_{j=1}^n (k a_{ij}+\sum_{i_0'\in I_0'} k_{ii_0'} a_{i_0'j}-\sum_{i_0''\in I_0''} k_{ii_0''} a_{i_0''j}) (u_j-u_j')y_j= 0\big)\\
&= \E_z\big(\sum_{j=1}^n (ka_{ij}+\sum_{i_0'\in I_0'} k_{ii_0'} a_{i_0'j}-\sum_{i_0''\in I_0''} k_{ii_0''} a_{i_0''j})z_j =0\big) 
\end{align*}

where  $z_j:=(u_j-u_j')y_j$, and in the last inequality we used the simple observation that $\E_{u,u'}(f(u)=0,f(u')=0) \le \E_{u,u'}(f(u)-f(u')=0)$.

Note that $u_j-u_j'$ and $y_j$ are iid copies of $\eta^{1/2}$. Hence $z_j$ are iid copies of $\eta^{1/4}$.

Finally, by Theorem \ref{theorem:ILO:general}, the bound 

$$n^{-O_{C,\ep}(1)}\le \E_z\big(\sum_{j=1}^n (ka_{ij}+\sum_{i_0'\in I_0'} k_{ii_0'} a_{i_0'j}-\sum_{i_0''\in I_0''} k_{ii_0''} a_{i_0''j})z_j =0\big)$$

implies that there exists a proper symmetric GAP $Q_i$ of rank $O_{C,\ep}(1)$ and size $n^{O_{C,\ep}(1)}$ such that the following holds for all but at most $n'$ elements of $j$

$$ka_{ij}+ \sum_{i_0'\in I_0'} k_{ii_0'} a_{i_0'j}-\sum_{i_0''\in I_0''} k_{ii_0''} a_{i_0''j} \in Q_i.$$

We summarize below.

\begin{theorem}[Refined row relation]\label{theorem:quadratic:rowrelation} Let $\ep<1$ and $C$ be positive constants.
Assume that $\rho_q(A)\ge n^{-C}$. Then there exist a set $I_0$ of size $O_{C,\ep}(1)$, a set $I$ of size at least $n-2n^\ep$, a number $0\neq k=n^{O_{C,\ep}(1)}$ such that for any $i\in I$ there are integers $k_{ii_0},i_0\in I_0$, all bounded by $n^{O_{C,\ep}(1)}$, an index set $J_i$ of size $n-n^\ep$, and a proper symmetric GAP $Q_i$ of rank $O_{C,\ep}(1)$ and size $n^{O_{C,\ep}(1)}$ such that the following holds for all $j\in J_i$

$$ka_{ij} + \sum_{i_0\in I_0} k_{ii_0} a_{i_0j}\in Q_i.$$

\end{theorem}

Clearly, we may assume that $I\cap I_0=\emptyset$ by throwing away those $i$ from $I$ that also belong to $I_0$. 

Let $R$ be the matrix defined below.

\begin{definition}[row matrix]\label{definition:quadratic:R} 
$R$ is an $n$ by $n$ matrix, whose $i$-th row, where $i\in I$, is defined by 

\begin{equation}\label{eqn:quadratic:R}
\row_i(R)(j):= 
\begin{cases}
k_{ij}, & \text{ if $j\in I_0$;}\\ 
k, & \text{ if $j=i$;}\\
0, & \text{ otherwise}
\end{cases}
\end{equation}

The other entries of $R$ are zero except the diagonal terms which are set to be 1.
\end{definition}

We restate Theorem \ref{theorem:quadratic:rowrelation} in a more convenient way below.

\begin{theorem}[Refined row relation, again]\label{theorem:rowrelation:quadratic}
Let $\ep \le 1$ and $C$ be positive constants. Assume that $\rho_q(A)\ge n^{-C}$. Then there exist a set $I_0$ of size $O_{C,\ep}(1)$, a set $I$ of size at least $n-2n^\ep$  satisfying $I\cap I_0=\emptyset$, integers $0\neq k,k_{ii_0},i_0\in I_0,i\in I$, all bounded by $n^{O_{C,\ep}(1)}$, and a matrix $R$ defined by \eqref{eqn:quadratic:R} such that the matrix $A'=RA$ possess the following properties: for each $i\in I$, there exist  a subset $\row_i' \subset \row_i(A')$ of size $n-n^\ep$ and a proper symmetric GAP $Q_i$ of rank $O_{C,\ep}(1)$ and size $n^{O_{C,\ep}(1)}$ such that $\row_i' \subset Q_i$.
\end{theorem}

Next, because $A$ is symmetric, we obtain a similar relation between the columns of $A$. Hence, we obtain the following key result.

\begin{theorem}[Matrix relations]\label{theorem:matrixrelation:quadratic}
Let $\ep \le 1$ and $C$ be positive constants. Assume that $\rho_q(A)\ge n^{-C}$. Then there exist a set $I_0$ of size $O_{C,\ep}(1)$, a set $I$ of size at least $n-2n^\ep$ satisfying $I\cap I_0=\emptyset$, integers $0\neq k,k_{ii_0},i_0\in I_0,i\in I$, all bounded by $n^{O_{C,\ep}(1)}$, and a matrix $R$ defined by \eqref{eqn:quadratic:R} such that the matrix $A'=RAR^T$ possess the following properties.

\begin{itemize}

\item For each $i\in I$, there exist  a subset $\row_i' \subset \row_i(A')$ of size $n-n^\ep$ and a proper symmetric GAP $Q_i$ of rank $O_{C,\ep}(1)$ and size $n^{O_{C,\ep}(1)}$ such that $\row_i' \subset Q_i$.

\vskip .1in

\item For each $j\in I$, there exist a subset $\col_j' \subset \col_j(A')$ of size $n-n^\ep$ and a proper symmetric GAP $P_j$ of rank $O_{C,\ep}(1)$ and size $n^{O_{C,\ep}(1)}$ such that $\col_j' \subset P_j$.

\end{itemize}
\end{theorem}

We complete the proof of Theorem \ref{theorem:ILO:quadratic:differentGAP} by using \eqref{eqn:aa'}, and Theorem \ref{theorem:ILO:quadratic:commonGAP} by using Lemma \ref{lemma:commonGAP}, noting that $k_{ij}=k_{ji}$ and $a_{ii_0}=a_{i_0i}$.

\begin{remark}
In later application we will not need the whole strength of Theorem \ref{theorem:matrixrelation:quadratic}. It will suffice to apply Theorem \ref{theorem:rowrelation:quadratic}.
\end{remark}

\section{proof of Lemma \ref{lemma:normalvector:linear}}\label{section:normalvector:linear}

We now prove Lemma \ref{lemma:normalvector:linear} by using our inverse Littlewood-Offord result for linear forms presented in Section \ref{section:linearILO}.

First of all, because $\rank(M_{n-1})=n-2$, the cofactor matrix $(a_{ij})$ of $M_{n-1}$ has rank $1$. Because this matrix is symmetric,  each entry $a_{ij}$ must have the form $a_{i}a_{j}$, where not all the $a_{i}$ are zeros. 

We will show that the vector $u=(a_1,\dots,a_{n-1})$ satisfies the conclusions of Lemma \ref{lemma:normalvector:linear}. 

Observe that 

$$\det(M_n) = \sum_{1\le i,j \le n-1} a_{ij}x_ix_j = (\sum_{i=1}^{n-1}a_ix_i)^2.$$ 

Thus the assumption $\P(\det(M_n)=0|M_{n-1})\ge n^{-C}$ implies that 

$$\P(\sum_{i=1}^{n-1}a_ix_i=0|M_{n-1})\ge n^{-C}.$$

By Theorem \ref{theorem:ILO}, all but $n^\ep$ elements of $a_i$ belong to a proper symmetric GAP of rank $O_{C,\ep}(1)$ and size $n^{O_{C,\ep}(1)}$. Also, by the definition of the $a_i$, $u=(a_1,\dots,a_{n-1})$ is orthogonal to $n-2$ linearly independent rows of $M_{n-1}$. We finish the proof of Lemma \ref{lemma:normalvector:linear} by using the following lemma.

\begin{lemma}[Rational commensurability]\label{lemma:rationalcommensurability}
Let $v=(v_1,\dots, v_{n-1})$ be a vector such that all but $n^\ep$ components $v_i$ belong to a proper symmetric GAP of rank $O_{C,\ep}(1)$ and size $n^{O_{C,\ep}(1)}$, and that $v$ is a normal vector of a hyperplane spanned by vectors of integral components bounded by $n^{O_{C,\ep}(1)}$. Then $\{v_1,\dots,v_{n-1}\} \subset \{(p/q)v_{i_0}, |p|,|q|=n^{O_{C,\ep}(n^\ep)}\}$ for some $i_0$. \end{lemma}

\begin{proof}(of Lemma \ref{lemma:rationalcommensurability})
Without loss of generality, we assume that $(v_{n-n^\ep},\dots, v_{n-1})$ are the exceptional elements that may not belong to the GAP. 

For each $v_i$, where $i< n-n^\ep$, there exist numbers $v_{ij}$, all bounded by $n^{O_{C,\ep}(1)}$, such that 

$$v_i = v_{i1}g_1+\dots v_{ir}g_r,$$ 

where $g_1,\dots, g_r$ are the generators of the GAP.

Note that by Theorem \ref{theorem:fullrank}, one may assume that the vectors $(v_{i1},\dots,v_{ir})$, where $i<n-n^\ep$, generate the whole space $\R^r$. 

Consider the $n-1$ by $r+n^\ep$ matrix $M_v$ whose $i$-th row is the vector $(v_{i1},\dots,v_{ir},0,\dots,0)$ if $i< n-n^\ep$, and $(0,\dots,0,1,0,\dots,0)$ if $n-n^\ep\le i$. Note that $M_v$ has rank $r+n^\ep$. 

We thus have  

$$v^T=M_v \cdot u^T,$$

where $u=(g_1,\dots,g_r,v_{n-n^\ep},\dots,v_{n-1})$.

Next, let $w_1,\dots,w_{n-2}$ be the vectors of integral entries bounded by $n^{O_{C,\ep}(1)}$ which are orthogonal to $v$. We form an $n-1$ by $n-1$ matrix $M_w$ whose $i$-th row is $w_i$ for $i\le n-2$, and the $n-1$-th row is $e_{i_0}$, a unit vector among the standard basis $\{e_1,\dots, e_{n-1}\}$ that is linearly independent to $w_1,\dots.w_{n-2}$. 

By definition, we have $M_w v^T = (0,\dots,0,v_{i_0})^T$, and hence 

$$(M_wM_v) u^T =(0,\dots,0,v_{i_0})^T.$$ 

The identity above implies that

\begin{equation}\label{eqn:smallvector:1}
(M_{w}M_{v}) (\frac{1}{v_{i_0}}u)^T =(0,\dots,0,1)^T.
\end{equation}

Next we choose a submatrix $M$ of size $r+n^\ep$ by $r+n^\ep$ of $M_{w}M_{v}$ thas has full rank. Then

\begin{equation}\label{eqn:smallvector:2}
M  (\frac{1}{v_{i_0}}u)^T = x
\end{equation}

for some $x$ which is a subvector of $(0,\dots,0,1)$ from \eqref{eqn:smallvector:1}.

Observe that the entries of $M$ are integers bounded by $n^{O_{C,\ep}(1)}$. Hence, the entries of $M^{-1}$ are fractions whose numerators and denominators are integers bounded by $(n^{O_{C,\ep}(1)})^{r+n^\ep}=n^{O_{C,\ep}(n^\ep)}$.

Solving for $ g_i/v_{i_0}$ and $v_{j}/v_{i_0}$ from \eqref{eqn:smallvector:2}, we conclude that each of these components can be written in the form $p/q$, where $|p|,|q|=n^{O_{C,\ep}(n^\ep)}$. 

\end{proof}

\begin{remark}
In principle, Lemma \ref{lemma:rationalcommensurability} is similar to Theorem 5.2 of \cite{TVsing}.
\end{remark}

\section{Proof of Lemma \ref{lemma:normalvector:quadratic}}\label{section:normalvector:quadratic}

In this section we will apply the results from Section \ref{section:quadraticILO} and Section \ref{section:proof:quadraticILO} to prove Lemma \ref{lemma:normalvector:quadratic}.

First, assume that 

$$\P_x(\sum_{ij}a_{ij}x_ix_j=0|M_{n-1})\ge n^{-C}.$$ 

Let $A$ be the matrix $(a_{ij})$. Then by Theorem \ref{theorem:ILO:quadratic:differentGAP} (or more explicitly, Theorem \ref{theorem:rowrelation:quadratic}), there exists a non-singular matrix $R$ (see Definition \ref{definition:quadratic:R}) such that most of the entries of each row of $A'$ belong to proper symmetric GAPs of small ranks and small sizes, where $A'=RA$.  

Set 

$$M:= M_{n-1}R^{-1}.$$

Because $M_{n-1}A = \det(M_{n-1})\cdot I_{n-1} \neq 0$, we have

\begin{equation}\label{eqn:quadraticILO:proof:1}
MA'=  M_{n-1}A= \det(M_{n-1}) \cdot I_{n-1} \neq \mathbf{O}.
\end{equation}

Next, it follows from the definition of $R$ that 

\[R^{-1}(ij)= 
\begin{cases}
1/k & \text{ if $i\in I$ and $j=i$};\\
-k_{ij}/k & \text{ if $i\in I$ and $j\in I_0$};\\
1 & \text{ if $i\notin I$ and $j=i$};\\
0 & \text{ if $i\notin I$ and $j\neq i$}.
\end{cases}
\]

We thus have, by $M(ij)= \sum_{j'} M_{n-1}(ij')(R^{-1})(j'j)$, that

\begin{eqnarray}
M(ij_0) &=& \sum_{j'\in I} M_{n-1}(ij')(-k_{j'j_0}/k)+M_{n-1}(ij_0), \mbox{ if } j_0\in I_0; \nonumber \\
M(ij)   &=& M_{n-1}(ij)/k, \mbox{ if } j\in I; \nonumber \\ 
\nonumber \\
M(ij)   &=& M_{n-1}(ij), \mbox{ if } j\notin I_0\cup I.
\end{eqnarray}

Because the entries of $M_{n-1}$ are $\pm 1$, and $k_{ii_0}=n^{O_{C,\ep}(1)}$, the entries of $M$ are rational numbers of the form $m/k$, where $m\in \Z$ and $m=n^{O_{C,\ep}(1)}$. Furthermore, note that $M$ also has full rank. 

Let $v$ be any column of $A'$ whose all but $n^\ep$ components belong to a proper symmetric GAP of rank $O_{C,\ep}(1)$ and size $n^{O_{C,\ep}(1)}$.

Because $MA' = \det(M_{n-1})\cdot I_n \neq \mathbf{0}$, $v$ is not a zero vector which is orthogonal to $n-2$ rows of $M$. Hence, it follows from Lemma \ref{lemma:rationalcommensurability} that  $\{v_1,\dots,v_{n-1}\} \subset \{(p/q)v_{i_0}, |p|,|q|=n^{O_{C,\ep}(n^\ep)}\}$ for some $i_0$.

Next, consider a row $\row_i(M)$ that is orthogonal to $v$, where $i\in I$. Note that there are at least $|I|-1 \ge n-2n^\ep-1$ such indices $i$. We have 

\begin{eqnarray}\label{eqn:normalvector:quadratic:1}
&&\sum_{j}M(ij)v_j = \sum_{j_0\in I_0}M(ij_0)v_{j_0}+\sum_{j\in I}M(ij)v_j + \sum_{j\notin I_0\cup I} M(ij)v_j \nonumber \\
&=&-\sum_{j_0\in I_0}\sum_{j'\in I}M_{n-1}(ij')k_{j'j_0}v_{j_0}/k+\sum_{j_0\in I_0}M_{n-1}(ij_0)v_{j_0}+\sum_{j\in I}M_{n-1}(ij)v_j/k + \sum_{j\notin I_0\cup I}M_{n-1}(ij)v_j \nonumber \\
&=& \sum_{j'\in I}M_{n-1}(ij')(-\sum_{j_0\in I_0}k_{j'j_0}v_{j_0}/k)+\sum_{j_0\in I_0} M_{n-1}(ij_0)v_{j_0}+ \sum_{j\in I}M_{n-1}(ij)v_j/k+\sum_{j_0\notin I_0\cup I}M_{n-1}(ij)v_j \nonumber \\
&=& \sum_{j\in I}M_{n-1}(ij)(v_j/k-\sum_{j_0\in I_0}k_{jj_0}v_{j_0}/k)+\sum_{j_0\in I_0}M_{n-1}(ij_0)v_{j_0}+\sum_{j\notin I_0\cup I}M_{n-1}(ij)v_j \nonumber \\
&=& 0.
\end{eqnarray}

Define

\[ u_j:=
\begin{cases}
v_j  & \text{ if $j\notin I$} ;\\\\
v_j/k -\sum_{j_0\in I_0} k_{jj_0}v_{j_0}/k & \text{ if $j\in I$}.
\end{cases}
\]

It then follows from \eqref{eqn:normalvector:quadratic:1} that 

$$\sum_{j}M_{n-1}(ij)u_j=0.$$

Thus, the vector $u=(u_1,\dots,u_{n-1})$ is orthogonal to $\row_i(M_{n-1})$. This holds for at least $n-2n^\ep-1$ rows of $M_{n-1}$.

Additionally, by the definition of $u$ and $v$,  all but $n^\ep$ coordinates of $u$ belong to a proper symmetric GAP of rank $O_{C,\ep}(1)$ and size $n^{O_{C,\ep}(1)}$ (with probably worse parameters), and  $\{u_1,\dots,u_{n-1}\} \subset \{(p/q)u_{j_0}, |p|,|q|=n^{O_{C,\ep}(n^\ep)}\}$ for some $j_0$. 

We conclude the proof by noting that, because $v$ is not a zero vector, $u$ is not either.

\section{proof of Theorem \ref{theorem:singularity:case1} and Theorem \ref{theorem:singularity:case2}}\label{section:singularity:completion}

Assume that $M_{n-1}$ has rank $n-2$ or $n-1$, and $\P(\det(M_n)=0|M_{n-1}) \ge n^{-C}$. We apply Lemma \ref{lemma:normalvector:linear} and \ref{lemma:normalvector:quadratic} to obtain a vector $u=(u_1,\dots,u_{n-1})$ of the following properties. 

\begin{enumerate}
\item All but $n^\ep$ elements of $u_i$ belong to a proper symmetric GAP of rank $O_{C,\ep}(1)$ and size $n^{O_{C,\ep}(1)}$. 

\vskip .1in

\item $u_i \in \{p/q: |p|,|q|=n^{O_{C,\ep}(n^\ep)}\}$ for all $i$.

\vskip .1in

\item $u$ is orthogonal to $n-O_{C,\ep}(n^\ep)$ rows of $M_{n-1}$.
\end{enumerate}

Let $\mathcal{P}$ denote the collection of all $u$ having property (1) and (2). For each $u\in \mathcal{P}$, let $\P_u$ be the probability, with respect to $M_{n-1}$, that $u$ is orthogonal to $n-O_{C,\ep}(n^\ep)$ rows of $M_{n-1}$. We shall prove the following key result.

\begin{theorem}\label{theorem:total} We have
$$\sum_{u\in \mathcal{P}}\P_u = O_{C,\ep}((1/2)^{(1-o(1))n}).$$
\end{theorem}

It is clear that Theorem \ref{theorem:singularity:case1} and Theorem \ref{theorem:singularity:case2} follow from Theorem \ref{theorem:total}.

In the sequel we will choose $0<\delta$ to be small enough so that $\delta \cdot O_{C,\ep}(1)\le \ep/4$ for all constants $O_{C,\ep}(1)$ appearing in Lemma \ref{lemma:normalvector:linear} and Lemma \ref{lemma:normalvector:quadratic}. 

Let $n_u$ denote the number on nonzero components of $u$.  To prove Theorem \ref{theorem:total} we decompose the sum $\sum_{u\in \mathcal{P}}\P_u $ into two parts depending on the magnitude of $n_u$.

\begin{theorem}\label{theorem:total:manyzeros}
The probability of a random symmetric matrix $M_{n-1}$ having $n-O_{C,\ep}(n^\ep)$ rows being orthogonal to a vector $u\in \mathcal{P}$ having $n_u\le n^{1-\delta}$ is bounded by 

$$\sum_{u\in \mathcal{P}, n_u\le n^{1-\delta}}\P_u =O\left((1/2)^{(1-o(1))n}\right),$$

where the implied constants depend on $C,\ep$ and $\delta$. 
 
\end{theorem}

\begin{theorem}\label{theorem:total:manynonzeros}
The probability of a random symmetric matrix $M_{n-1}$ having $n-O_{C,\ep}(n^\ep)$ rows being orthogonal to a vector $u\in \mathcal{P}$ having $n_u\ge n^{1-\delta}$ is bounded by

$$\sum_{u\in \mathcal{P}, n_u\ge n^{1-\delta}}\P_u =O(n^{-n^{1-\delta}/32}),$$ 
 
where the implied constants depend on $C,\ep$ and $\delta$. 

\end{theorem}

\begin{proof}(of Theorem \ref{theorem:total:manyzeros}) By paying a factor $\binom{n-1}{O_{C,\ep}(n^\ep)}=O(n^{O_{C,\ep}(n^\ep)})$ in probability, we may assume that the first $n-O_{C,\ep}(n^\ep)$ rows of $M_{n-1}$ are orthogonal to $u$. 

Also, by paying a factor $\binom{n}{n_u}$ in probability, we may assume that the first $n_u$ components of $u$ are nonzero. Thus we have

$$\sum_{i=1}^{n_u} u_{i}\row_{i}(M_{n-1})=0.$$

Which in turn implies that $\row_{n_u}(M_{n-1})$ lies in the subspace spanned by $\row_{1}(M_{n-1}),\dots,\row_{n_u-1}(M_{n-1})$.

Next, due to symmetry, $\row_{n_u}(M_{n-1})$ has $n_u-1$ components that were already exposed in the first $n_u-1$ rows (if we work with the general case that the rows in consideration are not necessarily the first $n_u$ rows of $M_{n-1}$, then there are less dependencies: at most  $n_u-1$ components already exposed in the previous $n_u-1$ rows.) 

Let $\row_{n_u}'$ be the subvector obtained from $\row_{n_u}$ by removing the exposed components, and for each $1\le i \le n_u-1$ we let $\row_{i}'$ be the subvectors of $\row_i(M_{n-1})$ corresponding to the columns restricted by $\row_{n_u}'$. 

By definition, each $\row_i'$ has $n-n_u$ components, and because $\row_{n_u}(M_{n-1})$ lies in the subspace spanned by $\row_{1}(M_{n-1}),\dots,\row_{n_u-1}(M_{n-1})$, so does $\row_{n_u}'$ in the subspace spanned by $\row_1',\dots, \row_{n_u-1}'$.  The probability for this event, by Lemma \ref{lemma:Odlyzko}, is at most 

$$2^{n_u-1-(n-n_u)}=2^{2n_u-n-1}.$$ 

Thus we have

$$\sum_{u\in \mathcal{P}, n_u\le n^{1-\delta}}\P_u \le \sum_{n_u=1}^{n^{1-\delta}} n^{O_{C,\ep}(n^\ep)} \binom{n}{n_u} 2^{2n_u-n} = O\left((1/2)^{(1-o(1))n}\right),$$

where the implied constants depend on $C,\ep$ and $\delta$. 

\end{proof}

\begin{remark}
In the proof of Theorem \ref{theorem:total:manyzeros}, because the assumption that $u$ has many zero components is strong, we do not need the additional additive structure on the remaining components of $u$. 
\end{remark}

We next focus on the estimate for the minor term.

\begin{proof}(of Theorem \ref{theorem:total:manynonzeros}) By paying a factor of $\binom{n-1}{n_u} \binom{n_u}{n^\ep}$ in probability and without loss of generality, we may assume that $u$ has the following properties: 

\begin{itemize}
\item the first $n_u$ components of $u$ are nonzero;

\vskip .1in

\item the first $n_0:=n_u-n^\ep$ components of $u$ are non-exceptional (that is they all belong to a proper symmetric GAP of rank $O_{C,\ep}(1)$ and size $n^{O_{C,\ep}(1)}$.)
\end{itemize}

 Because $u$ is orthogonal to $n-O_{C,\ep}(n^\ep)$ rows of $M_{n-1}$, it is orthogonal to $n_1:=n_0-O_{C,\ep}(n^\ep)$ rows among the first $n_0$ rows of $M_{n-1}$. By paying a factor of $\binom{n_0}{O_{C,\ep}(n^\ep)}= O(n_u^{O_{C,\ep}(n^\ep)})$ in probability, we may assume that these are the first $n_1$ rows of $M_{n-1}$. (One proceed similarly in the general case, occasionally with better bounds due to more independence among the entries.)

We will expose the first $n_1$ rows of $M_{n-1}$ one by one. Let $i$ be an index between $1$ and $n_1$. Condition on the first $i-1$ rows of $M_{n-1}$, the probability that $\row_i(M_{n-1})$ is orthogonal to $u$ is controlled by 

$$\P_{x_i,\dots,x_{n_u}\in \{-1,1\}}(\sum_{j=i}^{n_u} x_ju_j =-\sum_{j=1}^{i-1}x_ju_j) \le$$ 

$$\le \rho_i(u):=\sup_{a\in \R}\P_{x_i,\dots, x_{n_0}\in \{-1,1\}}(\sum_{j=i}^{n_0} x_ju_j=a).$$  

Observe that $\rho_1(u)\le \dots \le \rho_{n_1}(u)$. With room to spare, we concentrate on $\rho_i(u)$ where $i\le (1-\delta)n_0$ only. 

Note that $(1-\delta)n_0 <n_1$, thus the probability that the first $n_1$ rows of $M_{n-1}$ are orthogonal to $u$ is bounded by

\begin{equation}\label{eqn:singularity:orthogonalbound}
\P \left( \row_i(M_{n-1}), 1\le i \le (1-\delta)n_0, \mbox{ are orthogonal to } u \right) \le \prod_{i=1}^{(1-\delta)n_0} \rho_i.
\end{equation}

Note that the nonzero $u_j, j=1,\dots,n_0$, all belong to a proper symmetric GAP of rank $O_{C,\ep}(1)$ and size $n^{O_{C,\ep}(1)}$. It thus follows from the Erd\H{o}s-Littlewood-Offord inequality \eqref{eqn:Erdos} and Example \ref{example:linear:GAP}) that for any $1\le i \le (1-\delta)n_0 $   

\begin{equation}\label{eqn:singularity:concentrationbound}
 n^{-O_{C,\ep}(1)} \le \rho_i(u)=O((\delta n_0)^{-1/2}) = O((\delta n_u)^{-1/2}).
\end{equation}

Next we fix a sequence $b_0, b_1,\dots,b_K$, where $b_0=n^{-O_{C,\ep}(1)}$ is the left bound of \eqref{eqn:singularity:concentrationbound} and $b_{i+1}:=n^{\delta} b_i$, as well as $K$ is the smallest index such that $b_K$ exceeds the right bound of \eqref{eqn:singularity:concentrationbound} (thus $K\le O_{C,\ep}(1) \delta^{-1}$).

By the definition of the sequence $b_i$, for any $1\le i \le (1-\delta)n_0$ we have

$$b_0 \le \rho_i(u) \le b_K.$$

In the next step, we classify $u$ depending on how fast the concentration probabilities $\rho_i(u)$ grow.

\begin{definition}[concentration sequence]
We say that a $u\in \mathcal{P}$ satisfying $n_u \ge n^{1-\delta}$ has concentration sequence $(m_1,\dots, m_K)$, where $m_1+\dots+m_K=(1-\delta)n_0$,  if there are exactly $m_j$ terms $\rho_i(u)$ belonging to $[b_{j-1}, b_j)$. 

\end{definition}

Observe that the smaller $\delta$ we choose, the more detail we know about the distribution of $\rho_i(u)$.

Basing on concentration sequences, we say that $u \in \mathcal{P}$ belongs to $\mathcal{P}_{(m_1,\dots,m_K)}$ if its concentration sequence is $(m_1,\dots,m_K)$.

Our next lemma is to show that there is a collection of structures that contains all the elements of $\mathcal{P}_{(m_1,\dots,m_K)}$. This result will then enable us to compute $\P_u$ in a convenient way.

\begin{theorem}\label{theorem:concentrationsq}
Assume that $u\in \mathcal{P}_{(m_1,\dots,m_K)}$. Then there exists a sequence of proper symmetric GAPs $Q_0, Q_1,\dots, Q_{K}$ such that 

\begin{enumerate}

\item $u_i\in Q_0$ for all $1\le i \le n_0$;

\vskip .1in

\item $u_j \in Q_i$ for all but $n^\ep/K$ indices $j$ with $m_1+\dots + m_{i-1} \le j < m_1+\dots+ m_{i}$;

\vskip .1in

\item $|Q_i|\le c b_i^{-1}n^\delta/(n^\ep)^{1/2}$, where $c$ is a constant depending only on $C,\ep$ and $\delta$;

\vskip .1in

\item all the generators of $Q_i$ belong to the set $\{p/q, |p|,|q|\le n^{O_{C,\ep,\delta}(n^\ep)}\}$.

\end{enumerate}

\end{theorem}

Theorem \ref{theorem:concentrationsq} can be shown by applying Theorem \ref{theorem:ILO} several times. To begin with, we set $Q_0$ to be the proper symmetric GAP that contains all the non-exceptional $u$.  

Next, as 

$$\rho_{m_1+\dots+m_{i-1}} \ge b_{i-1} =O(n^{-O(1)})$$ 

Theorem \ref{theorem:ILO} implies that all but at most $n^\ep/K$ components $u_j$, where $m_1+\dots+m_{i-1}\le j\le (1-\delta)n_0$, belong to a proper symmetric GAP $Q_i$ of size  

\begin{align*}
O_{C,\ep,\delta}(\rho_{m_1+\dots+m_{i-1}}^{-1}/(n^\ep)^{1/2}) &=O_{C,\ep,\delta}(b_{i-1}^{-1} n^\delta/(n^\ep)^{1/2})\\
&= O_{C,\ep,\delta}(b_i^{-1} n^\delta/(n^\ep)^{1/2}).
\end{align*}

We keep this information only for those $u_j$ where $m_1+\dots + m_{i-1} \le j < m_1+\dots+ m_{i}$, and release other $u_j$ for the next application of Theorem \ref{theorem:ILO}. By Theorem \ref{theorem:fullrank}, we may assume that these $u_j$ span $Q_i$, and thus (4) holds because $u_j\in \{p/q, |p|,|q|=n^{O_{C,\ep,\delta}(n^\ep)}\}$.

Now for each $u\in \mathcal{P}_{(m_1,\dots, m_K)}$, we reconsider the probability that the first $n_1$ rows of $M_{n-1}$ are orthogonal to $u$. As shown in \eqref{eqn:singularity:orthogonalbound}, this probability is bounded by $\prod_{i}\rho_i$. By definition of concentration sequence, we have

\begin{equation}\label{eqn:singularity:probability}
\prod_{i=1}^{(1-\delta)n_0} \rho_i \le \prod_{i=1}^K b_i^{m_i}.
\end{equation}

In the next sequel we want to sum this bound over $u\in \mathcal{P}_{(m_1,\dots,m_K)}$.

Because each $Q_i$ is determined by its $O_{C,\ep,\delta}(1)$ generators from the set $\{p/q , |p|,|q|\le n^{O_{C,\ep,\delta}(n^\ep)}\}$, and its dimensions from the integers bounded by $n^{O_{C,\ep,\delta}(1)}$, there are $n^{O_{C,\ep,\delta}(n^\ep)}$ ways to choose each $Q_i$. So the total number of ways to choose $Q_0,\dots, Q_K$ is 

\begin{equation}\label{eqn:singularity:Q}
(n^{O_{C,\ep,\delta}(n^\ep)})^K = n^{O_{C,\ep,\delta}(n^\ep)}.
\end{equation}

Next, after locating $Q_i$, the total number of ways to choose is 

$$\prod_{i=1}^K\binom{m_i}{n^\ep/K} |Q_i|^{m_i-n^\ep/K} \le 2^{m_1+\dots+m_K}\prod_{i=1}^K|Q_i|^{m_i}=2^{(1-\delta)n_0}\prod_{i=1}^K|Q_i|^{m_i},$$

where $\binom{m_i}{n^\ep/k}|Q_i|^{m_i-n^\ep/K} $ is the number of ways to choose $u_j$ from each $Q_i$, following (2) of Theorem \ref{theorem:concentrationsq}. 

We then continue to estimate

\begin{eqnarray}\label{eqn:singularity:nonexceptional:1}
2^{(1-\delta)n_0}\prod_{i=1}^K|Q_i|^{m_i} &\le& (2c)^{(1-\delta)n_0} \prod_{i=1}^{K} (b_i^{-1}n^\delta/(n^\ep)^{1/2})^{m_i} \nonumber \\ 
&=& (2c)^{(1-\delta)n_0}\prod_{i=1}^K b_i^{-m_i} n^{\delta(1-\delta)n_0} n^{-\ep (1-\delta)n_0/2} \nonumber \\
&=& O(\prod_{i=1}^K b_i^{-m_i} n^{-\ep n_0/4}),
\end{eqnarray}

where in the last estimate we use the fact that $\delta \le \ep/16$.

For the remaining non-exceptional $u_i$, where $(1-\delta)n_0 \le i \le n_0$ or $u_j\notin Q_i$ from (2) of Theorem \ref{theorem:concentrationsq}, we choose them from $Q_0$, which results in the bound

\begin{equation}\label{eqn:singularity:nonexceptional:2}
b_0^{\delta n_0+n^\ep} = n^{\delta O_{C,\ep}(1) n_0}  \le n^{\ep n_0/8},
\end{equation}

where we use the fact that $\delta$ is chosen so that $\delta \cdot O_{C,\ep}(1) \le \ep/16$, and $n^\ep=o(n^{1-\delta})=o(n_0)$.

Regarding the exceptional elements $u_i$, where $n_0<i\le n_u$, we may choose them from  $\{p/q, |p|,|q|\le n^{O_{C,\ep}(n^\ep)}\}$, which results in the bound

\begin{equation}\label{eqn:singularity:exceptional}
(n^{O_{C,\ep}(n^\ep)})^{2n^\ep} = n^{O_{C,\ep}(n^{2\ep})}.
\end{equation}

Putting the estimates \eqref{eqn:singularity:Q}, \eqref{eqn:singularity:nonexceptional:1}, \eqref{eqn:singularity:nonexceptional:2} and \eqref{eqn:singularity:exceptional} together we obtain the bound for total number of ways to choose $u$ 

$$ n^{O_{C,\ep,\delta}(n^{2\ep})} n^{-\ep n_0/8} \prod_{i=1}^K b_i^{-m_i} \le O(n^{-\ep n_0/16})  \prod_{i=1}^K b_i^{-m_i},$$

where we use the fact that $n^{2\ep} =o(n^{1-\delta})=o(n_0)$.

Thus, according to \eqref{eqn:singularity:probability} we have

$$\sum_{u\in \mathcal{P}_{(m_1,\dots,m_k)}} \P \left( \row_i(M_{n-1}), 1\le i \le (1-\delta)n_0, \mbox{ are orthogonal to } u \right) = O(n^{-\ep n_u/16}).$$

Summing over the number of concentration sequences $(m_1,\dots,m_k)$ (which can be bounded cheaply by $n^K = n^{O_{C,\ep}(1)\delta^{-1}}$), over the positions of $n_u$ nonzero components and $n_0$ non-exceptional components of $u$ (which is bounded by $O(\binom{n-1}{n_u} \binom{n_u}{n^\ep})$), and over the position of $n_1$ rows of $M_{n-1}$ that are orthogonal to $u$ (which is bounded by $O(n_u^{O_{C,\ep}(n^\ep)})$), we hence obtain

$$\sum_{u\in \mathcal{P}, n_u\ge n^{1-\delta}}\P_u = O(n^{-\ep n^{1-\delta}/32}),$$

where the implied constant depends on $C,\ep$ and $\delta$, completing the proof.

\end{proof}

\vskip .2in

{\bf Acknowledgements.} The author would like to thank R. Pemantle and V. Vu for enthusiastic encouragement. He is grateful to the referees for careful reading the paper and providing very helpful remarks.


\end{document}